[1994/12/01]
\documentclass{amsart}
\pdfoutput=1
\chardef\bslash=`\\ 

\newtheorem{theorem}{Theorem}[section]
\newtheorem{corollary}[theorem]{Corollary}
\newtheorem{lemma}[theorem]{Lemma}
\newtheorem{proposition}[theorem]{Proposition}
\newtheorem{example}[theorem]{Example}

\newtheorem{remark}[theorem]{Remark}
\newtheorem{question}[theorem]{Question}

\theoremstyle{remark}

\newcommand{\B}{\mathcal{B}}

\usepackage{comment,amssymb,color}

\newcommand{\acts}{\curvearrowright}

\newcommand{\R}{\mathbb{R}}

\newcommand{\N}{\mathbb{N}}
\newcommand{\Z}{\mathbb{Z}}
\newcommand{\E}{\mathbb{E}}

\newcommand{\En}{\mathcal{H}}
\newcommand{\fEn}{\mathcal{H}^*}

\newcommand{\Prob}{\mathrm{Prob}}

\newcommand{\dist}{\mathrm{dist}}

\newcommand{\pull}{\Pi}

\newcommand{\Hom}{\mathrm{Hom}}

\newcommand{\Sym}{\mathrm{Sym}}

\newcommand{\Map}{\mathrm{Map}}

\newcommand{\mx}{\underline{x}}
\newcommand{\my}{\underline{y}}

\newcommand{\fx}{\mathbf{x}}
\newcommand{\fy}{\mathbf{y}}

\newcommand{\open}{\mathcal{O}}

\newcommand{\Appr}{\mathrm{Appr}}

\newcommand{\Pb}{\mathbb{P}}

\newcommand{\Cay}{G}

\newcommand{\past}{\mathcal{P}}

\usepackage{subfiles}

\usepackage{color}

\newcommand{\eval}[2][\right]{\relax
  \ifx#1\right\relax \left.\fi#2#1\rvert}

\usepackage{mathptmx}

\usepackage{xargs}                      
\usepackage[pdftex,dvipsnames]{xcolor}  
\usepackage[colorinlistoftodos,prependcaption,textsize=tiny]{todonotes}
\newcommandx{\unsure}[2][1=]{\todo[linecolor=red,backgroundcolor=red!25,bordercolor=red,#1]{#2}}
\newcommandx{\change}[2][1=]{\todo[linecolor=blue,backgroundcolor=blue!25,bordercolor=blue,#1]{#2}}
\newcommandx{\info}[2][1=]{\todo[linecolor=OliveGreen,backgroundcolor=OliveGreen!25,bordercolor=OliveGreen,#1]{#2}}
\newcommandx{\improvement}[2][1=]{\todo[linecolor=Plum,backgroundcolor=Plum!25,bordercolor=Plum,#1]{#2}}
\newcommandx{\thiswillnotshow}[2][1=]{\todo[disable,#1]{#2}}

\usepackage{cancel}
\usepackage{epigraph}

\usepackage{bbding}
\usepackage{bbm}

\usepackage{epigraph}
\usepackage{framed}
\usepackage[unicode=true, colorlinks = true, linkcolor=WildStrawberry, citecolor=NavyBlue, linktocpage=true, bookmarks=false]{hyperref}

\allowdisplaybreaks

\begin{document}

\title[Kieffer-Pinsker type formulas for Gibbs measures on sofic groups]{Kieffer-Pinsker type formulas\\for Gibbs measures on sofic groups}



\author[Raimundo Brice\~no]{Raimundo Brice\~no}
\address{Facultad de Matem\'aticas\\Pontificia Universidad Cat\'olica de Chile\\Santiago\\Chile}
\email{raimundo.briceno@mat.uc.cl}
\thanks{The author acknowledges the support of ANID/FONDECYT de Iniciación en Investigación 11200892.}

\subjclass[2010]{Primary 37A35, 37A60, 82B20; secondary 60B10, 37A15, 37A25, 37A50}


\date{}

\dedicatory{}

\begin{abstract}
Given a countable sofic group $\Gamma$, a finite alphabet $A$, a subshift $X \subseteq A^\Gamma$, and a potential $\phi: X \to \R$, we give sufficient conditions on $X$ and $\phi$ for expressing, in the uniqueness regime, the sofic entropy of the associated Gibbs measure $\mu$ as the limit of the Shannon entropies of some suitable finite systems approximating $\Gamma \acts (X,\mu)$. Next, we prove that if $\mu$ satisfies strong spatial mixing, then the sofic pressure admits a formula in terms of the integral of a random information function with respect to any $\Gamma$-invariant Borel probability measure with nonnegative sofic entropy. As a consequence of our results, we provide sufficient conditions on $X$ and $\phi$ for having independence of the sofic approximation for sofic pressure and sofic entropy, and for having locality of pressure in some relevant families of systems, among other applications. These results complement and unify those of Marcus and Pavlov (2015), Alpeev (2017), and Austin and Podder (2018).
\end{abstract}

\keywords{Countable group; sofic entropy; sofic pressure; variational principle; subshift; Gibbs measure; local weak* convergence; Shannon entropy; strong spatial mixing; phase transition; invariant random order; coupling}

\maketitle

\setcounter{tocdepth}{1}
\tableofcontents

\section{Introduction}

Sofic groups were introduced by Gromov \cite{gromov1999endomorphisms} and Weiss \cite{weiss2000sofic}. They include all amenable groups and residually finite groups, and it is an open question whether every countable group is sofic or not.

Since the seminal work of Lewis Bowen \cite{bowen2010measure}, there has been in intensive development of the study of dynamical invariants for the actions of countable sofic groups \cite{bowen2018brief}, including the definition of topological sofic entropy \cite{kerr2011entropy,kerr2013soficity} and sofic pressure \cite{1-chung}, among others \cite{1-bowen}. An important feature of these quantities is that, in contrast to the more classical versions for the amenable case, they may depend on the sofic approximation to the group and they could also take the value negative infinity. Thus, a relevant problem in this theory is to establish sufficient and necessary conditions for discriminating when we have such behavior. Recently, in \cite{1-airey}, Airey, Bowen, and Lin proved the existence of a topological dynamical system with two different positive topological sofic entropies. This is the first example of this kind and to find analogue examples in the measure-theoretic case (i.e., to find a measure preserving transformations with two different positive sofic entropies) remains an open problem.

The example in \cite{1-airey} was of a symbolic nature and, more specifically, proper $2$-colorings of a special kind of hypergraph that was susceptible of a delicate probabilistic analysis. This kind of context seems very promising for finding more relevant examples related to the aforementioned problems in entropy theory. Intimately related to these kind of systems are the so-called \emph{Gibbs measures}, which are Borel probability measures defined by prescribing some conditional expectations known as \emph{DLR equations} \cite{georgii2011gibbs}.

Gibbs measures are not an intrinsically dynamical object but they interact very nicely with them, and sometimes this turns out to be an advantage. In the amenable case, it is sufficient to have a single (not necessarily invariant) Gibbs measure in order to conclude, through a standard averaging argument, that there exists an invariant one. Moreover, also in the amenable case, and under mild mixing conditions on the support, there is a correspondence between the equilibrium states for a given potential and the associated Gibbs measures \cite{1-ruelle}. On the other hand, in the general sofic case, these questions are not entirely settled. For example, in \cite{alpeev2015entropy}, Alpeev managed to establish the existence of invariant Gibbs measures when the support is the full shift (similar results were also obtained by Grigorchuk and Stepin in \cite{grigorchuk1985gibbs} for the residually finite case), and recently, in \cite{1-shriver}, Shriver proved that in the full shift case, every free energy density minimizer is a Gibbs measure.

In parallel, on the one hand Alpeev in \cite{alpeev2017random}, and on the other hand, Austin and Podder in \cite{austin2018gibbs}, based on a statistical physics approach, established conditions for computing the sofic entropy of Gibbs measures by means of the Shannon entropy of some suitable sequence of finite models. They also proved a formula in terms of what Austin and Podder called \emph{percolative entropy}, a special kind of average conditional Shannon entropy with respect to some random subsets of the group involved, obtained through a percolation process. The results in \cite{alpeev2017random} and \cite{austin2018gibbs} have coincidences but also some differences. For example, in \cite{alpeev2017random}, Alpeev works with general sofic groups and the full shift, and in \cite{austin2018gibbs}, Austin and Podder work on trees ---which admit a particularly nice kind of sofic approximation--- but allowing hard constraints. A key condition that they both required for most of their results is that there is a unique Gibbs measure, in addition to correlation decay properties like \emph{strong spatial mixing} as in  \cite{austin2018gibbs} or related conditions like the \emph{Dobrushin's uniqueness criterion} or the presence of an \emph{attractive Gibbs structure} as in \cite{alpeev2017random}. In this work we focus our attention in the strong spatial mixing case, a property that has received attention in the last years, specially in the context of counting problems (e.g., see \cite{3-weitz}).

From another side, since the work of Hochman and Meyerovitch \cite{1-hochman}, in symbolic dynamics it has been important to establish conditions for having efficient algorithms to approximate the topological entropy of $\Z^d$ subshifts and more general groups. In order to tackle this problem, in \cite{1-marcus}, Marcus and Pavlov developed what they called an \emph{integral representation of pressure} in the $\Z^d$ context. One motivation for their results was to obtain a simpler representation of pressure in terms of the integral of the information function with respect to a simple measure, usually atomic for algorithmic purposes. Coincidentally, a sufficient condition for having such convenient representation is that the Gibbs measure involved satisfies strong spatial mixing plus some topological conditions on the support that, among other required consequences, provide control over the information function. A couple of years later, the topological conditions on the support were relaxed in \cite{1-briceno} by the introduction of the \emph{topological strong spatial mixing property}, a combinatorial property that interacts very well with the strong spatial mixing property and that in some cases is a necessary condition. Subshifts that satisfy the topological strong spatial mixing property are necessarily \emph{subshifts of finite type} but not vice versa and, in \cite{3-briceno}, it was fully characterized which set of constraints induce this property in the abstract context of relational structures.

In \cite{2-briceno}, there were established some extensions of the work of Marcus and Pavlov to the case of \emph{orderable} amenable groups. In \cite{1-alpeev}, Alpeev, Meyerovitch, and Ryu introduced in the amenable case what they called the \emph{Kieffer-Pinsker formula} for the Kolmogorov-Sinai entropy of a finite entropy partition by means of invariant random orders. Invariant random orders are a stochastic generalization of invariant (deterministic) orders in the context of orderable groups and an important advantage is that every countable group has at least one invariant random order. A canonical invariant random ordering of a group is the one induced by a percolation process, namely, to every element in $g \in \Gamma$ we attach i.i.d. random variables $\chi_g$ uniformly distributed in $[0,1]$, and for every realization of this process, we order the elements in $\Gamma$ according to the value of $\chi_g$. These kind of ideas were already present in the work of Kieffer \cite{kieffer1975generalized} and it was already observed in \cite{austin2018gibbs} the relationship between this order and percolative entropy. 

In this work, we build up on the results of Marcus and Pavlov, Alpeev, and Austin and Podder, and provide an extension and unification of them, with emphasis in the constrained case, i.e., proper subshifts. First, in Section \ref{sec2}, we introduce the basic definitions regarding sofic groups, sofic entropy, and symbolic dynamics. In Section \ref{sec3}, and assuming the topological strong spatial mixing property, we develop a formalism for studying subshifts by means of derived finite configuration spaces that approximate well the sofic action; in particular, in Theorem \ref{thm:partition}, we prove that the sofic pressure can be recovered as the limit superior of the normalized logarithm of the partition functions associated to such finite systems. Next, in Section \ref{sec4}, we introduce and study the properties of \emph{local weak* convergence}, a notion of convergence that is useful for comparing a sequence of measures supported on finite configuration spaces and a measure on a subshift; later, we define derived Gibbs measures supported on the derived finite configuration spaces and establish conditions for having local weak* convergence of them to a given Gibbs measure on a subshift. In Section \ref{sec5}, in Theorem \ref{thm:entropy}, we prove that when there is a unique Gibbs measure, the sequence of Shannon entropies obtained from the derived Gibbs measures converges to the sofic entropy; these results follow \cite{alpeev2017random}, but they are developed for proper subshifts satisfying the topological strong spatial mixing property and the techniques involved require a more careful analysis of entropy. We also prove in Corollary \ref{cor:eq} that, under the topological strong spatial mixing property and the uniqueness assumptions, Gibbs measures are necessarily equilibrium states. In Section \ref{sec6}, we introduce the strong spatial mixing property and establish some of its consequences. Then, in Section \ref{sec7}, given a invariant random past, we define the notion of \emph{ordered sofic approximation}; as an example, we prove that every sofic approximation can be ordered relative to the percolation order and that in the amenable case, the canonical sofic approximation obtained from a F{\o}lner sequence can be ordered relative to any invariant random past, thus providing a framework to relate the results in \cite{alpeev2017random,austin2018gibbs} with the results in \cite{1-marcus,2-briceno}, through the formalism introduced in \cite{1-alpeev}. In Section \ref{sec8}, we prove Theorem \ref{thm:main}, that we regard as the main result of our work, namely, if the support satisfies the topological strong spatial mixing property and the Gibbs measure satisfies strong spatial mixing, then the sofic pressure has an integral representation as in \cite{1-marcus}. In Section \ref{sec9}, we show how to recover the results in \cite{1-marcus,1-briceno,alpeev2017random,austin2018gibbs} from ours. As an application, in Corollary \ref{cor:indep} we provide sufficient conditions for having independence of the sofic approximation outside of the unconstrained case. Next, we show that, under some extra conditions, pressure turns out to be a local quantity, i.e., if a model is defined on two finitely generated groups whose Cayley graphs coincide on a large ball, then the associated sofic pressures are close in value. Finally, based on the techniques developed in \cite{5-briceno} for the amenable case, we outline a description of how to obtain a special representation of sofic pressure in terms of \emph{trees of self-avoiding walks} when the support has a safe symbol.


\section{Preliminaries}
\label{sec2}
%

\subsection{Sofic groups}

Let $\Gamma$ be a countable group with identity element $1_\Gamma$. Consider $\delta>0$ and $F \Subset \Gamma$, where $F \Subset \Gamma$ denotes that $F$ is a finite subset of $\Gamma$. Given a finite set $V$, we say that a map $\sigma: \Gamma \to \Sym(V)$ is 
\begin{enumerate}
\item[(i)] {\bf $(F, \delta)$-multiplicative} if $(1-\delta) |V| \leq |\{v \in V: \sigma^g(\sigma^h(v))=\sigma^{gh}(v)\}|$ for all $g, h \in F$, and
\item[(ii)] {\bf $(F, \delta)$-trace-preserving} if $\delta |V| \geq |\{v \in V: \sigma^g(v) = v\}|$ for all $g \in F \backslash\{1_\Gamma\}$,
\end{enumerate}
where $\sigma^g$ is just an alternative way to write $\sigma(g)$. Given a sequence of finite sets $\{V_n\}_n$ with $\lim_n\left|V_{n}\right|=\infty$, we say that a sequence of maps $\Sigma = \left\{\sigma_{n}: \Gamma \rightarrow \Sym\left(V_{n}\right)\right\}_n$ is a {\bf sofic approximation} to $\Gamma$ if for every $F \Subset \Gamma$ and $\delta>0$, there exists $n_0$ such that $\sigma_{n}$ is $(F, \delta)$-multiplicative and $(F, \delta)$-trace-preserving for every $n \geq n_0$. We say that a group $\Gamma$ is {\bf sofic} if it admits a sofic approximation (see \cite{gromov1999endomorphisms,weiss2000sofic}). We will assume that we have an arbitrary fixed sequence $\{F_r\}_r$ of finite subsets of $\Gamma$ such that $F_1 = \{1_\Gamma\}$, $F_r \subseteq F_{r+1}$, and $\bigcup_r F_r = \Gamma$. We call such a sequence an {\bf exhaustion}. Notice that, since $\{F_r\}_r$ is an exhaustion, it is enough to check along the sequence $\{F_r\}_r$ that $\Sigma$ is a sofic approximation, since the $(F, \delta)$-multiplicative and the $(F, \delta)$-trace-preserving properties are monotone in $F$.

\begin{example}
{\normalfont
Two important examples of sofic groups are the following.
\begin{itemize}
\item {\bf Amenable groups:} Given an amenable group $\Gamma$, a (left) F{\o}lner sequence $\{T_n\}_n$ for it, and $g \in \Gamma$, consider the permutation $\sigma_n^g: T_n \to T_n$ given by
$$
\sigma_n^g(h) = 
\begin{cases}
gh		&	\text{ if } h \in T_n \cap g^{-1} T_n,		\\
b_n(h)	&	\text{ if } h \in T_n \setminus g^{-1}T_n,
\end{cases}
$$
where $b_n: T_n \backslash g^{-1} T_n \rightarrow g T_n \backslash T_n$ is an arbitrary bijection. Then, the sequence $\Sigma = \left\{\sigma_n: \Gamma \to \Sym\left(T_n\right)\right\}_n$ is a sofic approximation to $\Gamma$, since $\lim_n |T_n|^{-1}|g T_n \setminus T_n| = 0$.

\item {\bf Residually finite groups:} If a group $\Gamma$ is residually finite, there exists a sequence $\left\{N_{n}\right\}_n$ of finite index normal subgroups such that $\bigcap_{i=1}^{\infty} \bigcup_{n=i}^{\infty} N_n = \{1_\Gamma\},$ i.e., every nontrivial group element is contained in only finitely many $N_n$. For each $n$ we define a group homomorphism $\sigma_{n}: \Gamma \to \Sym\left(\Gamma / N_{n}\right)$ by setting $\sigma_{n}^g\left(hN_{n}\right) = ghN_{n}$ for all $g, h \in \Gamma$. Then, for every $g_1, g_2 \in \Gamma$ and for every $n$, it holds that $\sigma_n^{g_1}(\sigma_n^{g_2}(hN_n))=\sigma_n^{g_1g_2} (hN_n)$ for all $g_1,g_2,h \in \Gamma$, and for every $g \in \Gamma \setminus \{1_\Gamma\}$ and for every sufficiently large $n$, it holds that $\sigma_n^g(hN_n) \neq hN_n$ for all $h \in \Gamma$. Thus, $\Sigma = \{\sigma_n: \Gamma \to \Sym\left(\Gamma / N_{n}\right)\}_n$ is a sofic approximation to $\Gamma$.
\end{itemize}
}
\end{example}

We will sometimes use the following nomenclature. If $\{\Omega_n\}_n$ is a sequence of finite sets and $\mathrm{P}$ is a property that holds for some elements of $\Omega_{n}$ for each $n$, then $\mathrm{P}$ holds {\bf with high probability (w.h.p.) in $\omega \in \Omega_n$} if
$$
\lim_n |\Omega_n|^{-1} \left\{\omega \in \Omega_{n}: P \text { holds for } \omega\right\} = 1.
$$

Then, a sofic approximation corresponds to a sequence $\Sigma=\left\{\sigma_{n}: \Gamma \rightarrow \Sym\left(V_{n}\right)\right\}_n$ such that for all $g, h \in \Gamma$,
$$
\sigma_{n}^{g}(\sigma_{n}^{h}(v)) = \sigma_{n}^{g h}(v) \quad \text { w.h.p. in } v \in V_n
$$
and, for all $g \in \Gamma \backslash\left\{1_{\Gamma}\right\}$,
$$
\sigma_{n}^{g}(v) \neq v  \quad	\text { w.h.p. in } v \in V_n.
$$

Given $F \subseteq \Gamma$, we denote $\sigma_n^F(v) := \{\sigma_n^g(v): g \in F\}$. We will say that $v \in V_n$ is {\bf $F$-good} for $F \Subset \Gamma$ if
\begin{enumerate}
\item $\sigma_n^g(v) \neq \sigma_n^h(v)$ for every $g,h \in F$ such that $g \neq h$;
\item $\sigma_n^g(\sigma_n^h(u))=\sigma_n^{gh}(u)$ for every $g, h \in F$ and $u \in \sigma_n^F(v)$;
\item $\sigma_n^{g^{-1}}(\sigma_n^g (u)) = u$ for every $u \in \sigma_n^F(v)$; and,
\item if $u \in \sigma_n^F(v)$, $w \in V_n$, and $g \in F$ are such that $u = \sigma_n^g(w)$, then $w = \sigma_n^{g^{-1}}(u)$.
\end{enumerate}

We denote by $V^F_n$ the subset of $F$-good vertices in $V_n$ and by $V^r_n$ the set of $F_r$-good vertices. It is not difficult to see that, for any $F \Subset \Gamma$, it holds that $\lim_n \frac{|V_n^F|}{|V_n|}=1$ (see \cite[Lemma 2.2]{alpeev2017random}).

\subsubsection{The finitely generated case}

If $\Gamma$ is finitely generated and $S \Subset \Gamma$ is a symmetric generating set, we define the {\bf (left) Cayley structure} of $\Gamma$ with respect to $S$ as
$$
\Cay(\Gamma, S) = \left<\Gamma;\{R_{s}(\Gamma)\}_{s \in S}\right>,
$$
with $R_s(\Gamma) = \{(g,sg): g \in \Gamma\}$ for $s \in S$. We see $\Cay(\Gamma, S)$ as a regular directed graph with edge labels $s \in S$. This induces the usual word metric $|\cdot|_S$ in $G(\Gamma, S)$ and, given $r \in \N$, we can define the {\bf $r$-ball centered at $g$} as $B(g, r) = \{h \in \Gamma: |h^{-1}g|_S \leq r\}$.

In this context, a sofic approximation $\Sigma = \{\sigma_n: \Gamma \to \Sym(V_n)\}_n$ induces a sequence of finite $S$-edge-labeled directed graphs $\{G_n\}_n$, where $G_n$ has vertex set $V_n$ and edge set $E_n = \{(v,\sigma^s(v)): v \in V_n, s \in S, \sigma^s(v) \neq v\}$. Notice that for $F \Subset \Gamma$ and $v \in V_n^F$, we have that $\Cay(\Gamma,S)[F]$ is isomorphic to $G_n[\sigma^F(v)]$ ---denoted by $\Cay(\Gamma,S)[F] \cong G_n[\sigma^F(v)]$---, where $G_n[U]$ denotes the graph induced by a subset of vertices $U$ and we require that the isomorphism maps $v$ to $1_\Gamma$, preserves edge directions, and preserves labels.

We sometimes abbreviate $B(1_G,r)$ by $B_r$. Notice that $\{B_r\}_r$ is an exhaustion of $\Gamma$ and that $V_n^{r} = \{v \in V_n: B(1_\Gamma,r) \cong B(v,r)\}$, where $B(v,r)$ corresponds to the $r$-neighborhood of $v$ in $G_n$. We say that a sequence $\left\{G_{n}\right\}_n$ of finite $S$-edge-labeled directed graphs {\bf Benjamini-Schramm converges} to $\Cay(\Gamma, S)$ if for all $r \in \N$ and $\delta>0$, there exists $n_0 \in \N$ such that $|V^r_n| \geq (1-\delta) |V_n|$ for all $n \geq n_0$. In other words, with high probability, the $r$-neighborhood of a vertex uniformly chosen at random in $V$ looks the same as the $r$-neighborhood of the identity in the Cayley diagram $\Cay(\Gamma, S)$ (e.g., see \cite{borgs2013left}).

Then, in the finitely generated case, a sofic approximation is equivalent to the existence of a sequence of $S$-edge-labeled directed graphs Benjamini-Schramm converging to $\Cay(\Gamma, S)$. Moreover, $|S|$ is an upper bound for the maximum out-degree of $G_n$ and, by maybe removing some edges and a negligible number of vertices, we can also assume that both the in-degrees and out-degrees are uniformly bounded in $n$ by $|S|$.


\subsection{Sofic pressure}

Now, assume $\Gamma$ is a sofic group, $\Sigma = \{\sigma_n: \Gamma \to \Sym(V_n)\}_n$ is a sofic approximation to $\Gamma$, and let $(X,\rho)$ be a compact metric space such that $\Gamma$ acts on $X$ by homeomorphisms, i.e., there is a collection $T=\left\{T^{g}\right\}_{g \in \Gamma}$ of homeomorphisms $T^{g}: X \rightarrow X$ satisfying $T^{g h}=T^{g} T^{h}$ and $T^{g^{-1}}=\left(T^{g}\right)^{-1}$. We denote such fact by $\Gamma \acts^T X$ or simply by $\Gamma \acts X$ if $T$ is understood.

Given $\epsilon > 0$, a subset $Y \subseteq X$ is said to be {\bf $(\rho, \epsilon)$-separated} if $\rho(x, y) \geq \epsilon$ for all $x, y \in Y$ with $x \neq y$ and we denote by $N_{\epsilon}(Y, \rho)$ the maximum cardinality of a $(\rho, \epsilon)$-separated subset of $X$. We let $\rho_{2}$ and $\rho_{\infty}$ denote the metrics on $X^{V_n}$ defined by
$$
\rho_{\infty}(\mx, \my) := \max_{v} \rho\left(x_{v}, y_{v}\right) \quad \text{ and } \quad \rho_{2}(\mx, \my) := \left(\frac{1}{|V|} \sum_{v} \rho\left(x_{v}, y_{v}\right)^{2}\right)^{1/2},
$$
where $\mx = \left(x_{1}, \ldots, x_{|V|}\right)$ and $\my =\left(y_{1}, \ldots, y_{|V_n|}\right)$ belong to $X^{V_n}$. Given $F \Subset \Gamma$, and $\delta>0,$ we define
$$
\Map(T, \rho, F, \delta, \sigma_n) := \{\mx \in X^{V_n}: \rho_{2}\left(T^{g} \mx, \mx \circ \sigma_n(g)\right)<\delta \text{ for all } g \in F\},
$$
where $\left(T^{g} \mx\right)_{v} = T^{g} x_{v}$ and $(\mx \circ \sigma_n(g))_{v} = x_{\sigma_n^g(v)}$ for $v \in V_n$. An element $\mx \in \Map(T, \rho, F, \delta, \sigma_n)$ is called a {\bf microstate}. Then, the {\bf topological sofic entropy} of $T$ with respect to $\rho$ and $\Sigma$ is defined as
$$
h_{\Sigma}(\Gamma \acts^T X) := \sup_{\epsilon>0} \inf_{F \Subset \Gamma} \inf_{\delta>0} \limsup_{n \to \infty}\left|V_n\right|^{-1} \log N_{\epsilon}\left(\Map\left(T, \rho, F, \delta, \sigma_{n}\right), \rho_{\infty}\right),
$$
which is a way of measuring the exponential growth rate of the number of approximate partial orbits that can be distinguished up to some small scale.

Now consider a continuous function $\phi: X \to \R$ that we call {\bf potential} and denote by $\|\phi\|$ the supremum norm of $\phi$. Given a microstate $\mx: V_n \to X$, we define its {\bf energy} by
$$
\En_{\sigma_n}(\mx) := \sum_{v \in V_n} \phi(x_v).
$$

For any finite collection of microstates $\mathcal{Z} \subseteq X^{V_n}$, we define the associated {\bf partition function} as
$$
Z_{\sigma_n}(\mathcal{Z}) := \sum_{\mx \in \mathcal{Z}} \exp (\En_{\sigma_n}(\mx)).
$$

Then, given $F \Subset \Gamma$ and $\delta, \epsilon>0$, we let
$$
Z_{\epsilon}(\phi,T, \rho, F, \delta, \sigma_n) := \sup\{Z_{\sigma_n}(\mathcal{Z}): \mathcal{Z} \subseteq \Map(T, \rho, F, \delta, \sigma_n) \text{ is $\left(\rho_{\infty}, \epsilon\right)$-separated}\}.
$$
and define the {\bf topological sofic pressure} of $T$ and $\phi$ with respect to $\rho$ and $\Sigma$ as
$$
p_{\Sigma}(\Gamma \acts^T X , \phi) := \sup_{\epsilon>0} \inf_{F \Subset \Gamma} \inf_{\delta>0} \limsup _{n \to \infty}\left|V_{n}\right|^{-1} \log Z_{\epsilon}\left(\phi,T, \rho, F, \delta, \sigma_{n}\right).
$$

Notice that, if $\phi \equiv 0$, then topological sofic pressure coincides with topological sofic entropy. Moreover, if $\Gamma$ is amenable, then this definition of topological sofic pressure coincides with the more usual one in terms of F{\o}lner sequences (see \cite[Theorem 1.1]{1-chung}).

A {\bf pseudometric} on $X$ is a function $\rho: X \times X \rightarrow[0, \infty)$ that is symmetric and satisfies the triangular inequality. A pseudometric $\rho$ on $X$ is said to be {\bf generating} for $\Gamma \acts^T X$ if for every $x, y \in X$ such that $x \neq y$, there exists $g \in \Gamma$ with $\rho(T^g x, T^g y) > 0$. It is known that topological sofic entropy and topological sofic pressure can be defined in terms of any generating pseudometric and the value is independent of the choice of it. Moreover, $\rho_{\infty}$ can be replaced by $\rho_{2}$ in the definition of entropy and pressure (e.g., see \cite[Proposition 10.23]{1-kerr}).

\subsubsection{Variational principle}

We denote by $\B$ the Borel $\sigma$-algebra and by $\Prob(X)$ the space of Borel probability measures on $X$. Given $\nu \in \Prob(X)$, we say that an action $T$ on $(X, \nu)$ is {\bf $\Gamma$-invariant} if each $T^{g}$ preserves $\nu$. We denote such an action by $\Gamma \acts^T (X, \nu)$ and by $\Prob(X,\Gamma)$ the subspace of $\Gamma$-invariant Borel probability measures on $X$.
 
Given a sequence $\{\nu_n\}_n$ in $\Prob(X)$, we say that it {\bf weak* converges} to a measure $\nu$ if, for every $f \in \mathcal{C}(X)$,
$$
\int f d \nu_{n} \to \int f d \nu	\quad	\text{ as } n \to \infty.
$$

This induces the so-called {\bf weak* topology} on $\Prob(X)$ and, by the Banach-Alaoglu Theorem, $\Prob(X)$ turns out to be compact when endowed with such topology. Given an open neighborhood $\open \subseteq \Prob(X)$, we define
$$
\Map(T, \rho, \open, F, \delta, \sigma_n) := \{\mx \in \mathrm{Map}(T, \rho, F, \delta, \sigma_n): |V_n|^{-1} \sum_{v \in V_n} \delta_{x_v} \in \open\},
$$ 
where $\delta_{x_v}$ is the Dirac delta supported at $x_v$. The set $\Map(T, \rho, \open, F, \delta, \sigma_n)$ corresponds to the microstates that are approximately equidistributed. Then, the {\bf sofic entropy} of $T$ and $\nu$ with respect to $\rho$ and $\Sigma$ is
$$
h_{\Sigma}(\Gamma \acts^T (X,\nu)) := \sup_{\epsilon>0} \inf_{\open \ni \nu} \inf_{F \Subset \Gamma} \inf_{\delta>0} \limsup_{n \rightarrow \infty} \left|V_n\right|^{-1} \log \left(N_{\epsilon}\left(\Map\left(T, \rho, \open, F, \delta, \sigma_{n}\right), \rho_{\infty}\right)\right).
$$

As in the topological case, it is known that $h_{\Sigma}(\Gamma \acts^T (X,\nu))$ coincides for any generating pseudometric $\rho$ (see \cite{1-kerr}). In \cite{1-chung}, Chung proved the following variational principle for pressure.

\begin{theorem}
Let $\Gamma \acts^T X$ be a continuous action on a compact metrizable space and $\phi: X \to \R$ a potential. Then, for every sofic approximation $\Sigma$,
$$
p_{\Sigma}(\Gamma \acts^T X, \phi) = \sup_{\nu \in \Prob(X,\Gamma)} \left\{h_{\Sigma}(\Gamma \acts^T (X,\nu)) + \int \phi d\nu\right\}.
$$
\end{theorem}

A measure $\mu \in \Prob(X,\Gamma)$ is called an {\bf equilibrium state} for $(\Gamma \acts^T X, \phi)$ if it realizes the supremum in the variational principle, i.e.,
$$
p_{\Sigma}(\Gamma \acts^T X, \phi) = h_{\Sigma}(\Gamma \acts^T (X,\mu)) + \int \phi d\mu.
$$

It is known \cite{2-chung}, that if the action $\Gamma \acts^T X$ is expansive, then the entropy is upper semi-continuous in the measure $\nu$, and therefore there exists an equilibrium state.

%

\subsection{Subshifts}

Given a finite set $A$, consider the space $A^\Gamma$ endowed with the product topology, that we call the {\bf full shift}, and $T$, the right-shift action of $\Gamma$ on $A^\Gamma$ given by
$$
T^{g}((a_{h})_{h \in \Gamma}) := (a_{h g})_{h \in \Gamma},
$$
where $x = (a_h)_{h \in \Gamma} \in A^\Gamma$. Notice that $T$ is expansive. A closed and shift-invariant subset $X \subseteq A^{\Gamma}$ is called a {\bf subshift}. From now on, we will suppose that $X$ is a subshift and $\{F_r\}_r$ is a fixed exhaustion of $\Gamma$. Let $\rho$ be the generating pseudometric given by
$$
\rho(x,y) =
\begin{cases}
1	&	\text{ if } x(1_\Gamma) \neq y(1_\Gamma),	\\
0	&	\text{ otherwise.}
\end{cases}
$$

Then, an easy calculation shows that $\rho_2\left(T^g \mx, \mx \circ \sigma_n(g)\right)^2 = |V_n|^{-1} |\{v \in V_n: x_v(g) \neq x_{\sigma_n^g(v)}(1_\Gamma)\}|$, so
$$
\Map(T, \rho, F, \delta, \sigma_n)	=	\{\mx \in X^{V_n}: |\{v \in V_n: x_v(g) = x_{\sigma_n^g(v)}(1_\Gamma)\}|	>	(1-\delta^{1/2})|V_n| \text{ for all } g \in F\}.
$$

Define the set
\begin{align*}
\Map(r, \delta, \sigma_n)	&	:=	\{\mx \in X^{V_n}: |\{v \in V_n: x_v(g) = x_{\sigma_n^g(v)}(1_\Gamma) \text{ for all } g \in F_r\}|	>	(1-\delta)|V_n|\}.
\end{align*}

Notice that $\Map(r, \delta^{1/2}, \sigma_n) \subseteq \Map(T, \rho, F_r, \delta, \sigma_n)$ and if $\mx \in \Map(T, \rho, F_r, \delta, \sigma_n)$, then
$$
|\{v \in V_n: x_v(g) = x_{\sigma^g(v)}(1_\Gamma) \text{ for all } g \in F_r\}|	 \geq (1-\delta|F_r|)|V_n|,
$$
so $\Map(r, \delta^{1/2}, \sigma_n) \subseteq \Map(T, \rho, F_r, \delta, \sigma_n) \subseteq \Map(r, \delta|F_r|, \sigma_n)$. In particular, this implies that, given a sofic approximation $\Sigma = \left\{\sigma_{n}: \Gamma \to \Sym(V_n)\right\}_n$ of $\Gamma$, we have that
\begin{align*}
h_{\Sigma}(\Gamma \acts  X)	&	= \sup_{\epsilon>0} \inf_{r, \delta>0} \limsup_{n \to \infty}\left|V_n\right|^{-1} \log N_{\epsilon}\left(\Map\left(r, \delta, \sigma_{n}\right), \rho_{\infty}\right)	\\
						&	= 	\inf_{r, \delta>0} \limsup_{n \to \infty}\left|V_n\right|^{-1} \log N_{1}\left(\Map\left(r, \delta, \sigma_{n}\right), \rho_{\infty}\right),
\end{align*}
where the last equality follows from the fact that $\rho_\infty$ only takes values $0$ and $1$, so we can assume that $\epsilon = 1$. Similarly, we can define
$$
Z(\phi, r, \delta, \sigma_n) :=  \sup\{Z_{\sigma_n}(\mathcal{Z}): \mathcal{Z} \subseteq \Map\left(r, \delta, \sigma_n\right) \text{ is $\left(\rho_{\infty}, 1\right)$-separated}\}
$$ 
to obtain that
$$
p_{\Sigma}(\Gamma \acts X, \phi) = \inf_{r,\delta > 0} \limsup_{n \to \infty}\left|V_{n}\right|^{-1} \log Z(\phi, r, \delta, \sigma_n).
$$

\section{Derived configuration spaces}
\label{sec3}

Given a sofic approximation $\Sigma = \{\sigma_n: \Gamma \to \Sym(V_n)\}_n$, we are interested in constructing finite models in $A^{V_n}$ that locally approximate $\Gamma \acts X$. For $F \subseteq \Gamma$ and $x \in A^\Gamma$, we denote by $x_F$ the restriction of $x$ to $F$. Considering this, we define $X_F = \{x_F: x \in X\}$ and $[x_F] = \{y \in X: y_F = x_F\}$. We say that $w \in A^F$ is {\bf globally admissible} if $[w] \neq \emptyset$. Given disjoint sets $E, F \subseteq \Gamma$, $w_E \in A^E$, and $w_F \in A^F$, we denote by $w_Ew_F$ the \emph{concatenation} of $w_E$ and $w_F$. Notice that $[(T^g x)_F] = T^g[x_{Fg}]$. 

\subsection{Pullback names}

Given $U \subseteq V_n$, we say that any $\fx = (a_v)_{v \in U} \in A^U$ is a {\bf partial configuration}. Given such partial configuration $\fx$, $F \subseteq \Gamma$, and $v \in V$, we define the {\bf $F$-pullback name of $\fx$ at $v$} as
$$
\pull_{v}^{\sigma_n,F}(\fx) :=  (a_{\sigma_n^g(v)})_{\{g \in F: \sigma_n^g(v) \in U\}}.
$$

If $U = V_n$ and $F = \Gamma$, we recover the usual notion of pullback name in the literature \cite{austin2018gibbs}. If $F = F_r$, we write $\pull_{v}^{\sigma_n,r}(\fx)$ instead. Notice that this defines a function $\pull^{\sigma_n}: A^{V_n} \to (A^\Gamma)^{V_n}$ given by
$$
\pull^{\sigma_n}(\fx) := (\pull_v^{\sigma_n}(\fx))_{v \in V_n}.
$$

We also introduce the map $\theta_{\sigma_n}: (A^\Gamma)^{V_n} \to A^{V_n}$ given by $\theta_{\sigma_n}(\mx)_v = x_v(1_\Gamma)$. Notice that if $\rho_\infty(\mx, \my) \geq 1$, then $\theta_{\sigma_n}(\mx) \neq \theta_{\sigma_n}(\my)$. Therefore, $\theta_{\sigma_n}$ is injective when restricted to any $(\rho_\infty,1)$-separated set.

\subsection{The topological strong spatial mixing property}

Given $M \Subset \Gamma$, we will say that a subshift $X$ has the {\bf topological strong spatial mixing property (TSSM) with range $M$} if for every $F \subseteq \Gamma$ and for every $w \in A^F$,
$$
w \in X_F	\iff	w_{F \cap M g} \in X_{F \cap M g}  \text{ for all } g \in F.
$$

We say that $X$ satisfies the {\bf  topological strong spatial mixing property (TSSM) property} if it has the TSSM property with range $M$ for some $M$. Without loss of generality, we will assume that $1_\Gamma \in M$ and $M = M^{-1}$.

\begin{remark}
Notice that the TSSM property, maybe after readjusting the range, is equivalent to the following a priori weaker condition:
$$
w \in X_F	\iff	w_{F \cap Mg} \in X_{F \cap Mg}  \text{ for all } g \in \Gamma.
$$

In other words, a subshift has the TSSM property if we can locally check the global admissibility of any pattern $w$. In particular, if a subshift satisfies the TSSM property, then it must be a \emph{subshift of finite type}, but the converse is not true (see \cite{1-briceno}). On the other hand, there are many well-known models supported on subshifts that satisfy the TSSM property, such as the full shift, independent sets, proper colorings with enough colors, among many others (e.g., see \cite{1-briceno,3-briceno,4-briceno}).
\end{remark}

\begin{lemma}
\label{lem:consist}
Suppose that $X$ satisfies the TSSM property with range $M$. Then, for every $F \subseteq \Gamma \setminus \{1_\Gamma\}$, $w \in A^F$, and $a \in A$, $wa^{1_\Gamma}$ is globally admissible if and only if $w$ and $w_{F \cap MM}a^{1_\Gamma}$ are globally admissible.
\end{lemma}

\begin{proof}
Let's assume that $w$ and $w_{F \cap MM}a^{1_\Gamma}$ are globally admissible. By the TSSM property, it suffices to check that  $(wa^{1_\Gamma})_{F \cap Mg}$ is globally admissible for every $g \in F \cup \{1_\Gamma\}$ in order to conclude that $wa^{1_\Gamma}$ is globally admissible. If $1_\Gamma \notin Mg$, then the global admissibility of  $(wa^{1_\Gamma})_{F \cap Mg}$ reduces to the global admissibility of $w_{F \cap Mg}$. If $1_\Gamma \in Mg$, this means that $g^{-1} \in M$, which is equivalent to $g \in M$, since $M$ is symmetric. Therefore, since $F \cap Mg \subseteq F \cap MM$, the global admissibility of $(wa^{1_\Gamma})_{F \cap Mg}$ reduces to the global admissibility of $w_{F \cap MM}a^{1_\Gamma}$. The converse is direct.
\end{proof}

We define the {\bf $n$th derived configuration space} $X^n$ of $X$ as
$$
X^n = \{\fx \in A^{V_n}: \pull^{\sigma_n,MM}_v(\fx) \in X_M \text{ for all } v \in V_n^{MM}\},
$$
where $MM = \{gh: g,h \in M\}$ and $V_n^{MM}$ is the set of $MM$-good vertices. In simple words, $X^n$ is a configuration space that, locally, respects the constraints of $X$ wherever it makes sense to apply them, with a few exceptions, namely, the set $V_n^{M} \setminus V_n^{MM}$ which is negligible compared to $V_n$. 

Given a subset $U \subseteq V_n$, we define $\fx_U$, $X^n_U$, etc., exactly as for subshifts and say that $\fx_U$ is {\bf locally consistent} if
$$
[\pull_{v}^{\sigma_n,MM}(\fx_U)] \neq \emptyset \quad	\text{ for all } v \in V_n^{MM}.
$$

In addition, for $F \subseteq \Gamma$, we denote $\sigma_n^{F}(U) = \{\sigma^g(u): g \in F, u \in U\}$. We have the following key lemma.

\begin{lemma}
\label{lem:extend}
If $X$ satisfies the TSSM property with range $M$, then, for every $U \subseteq V_n$ and $\fx_U \in A^U$ that is locally consistent, there exists a locally consistent extension $\fx \in X^n$ of $\fx_U$.
\end{lemma}

\begin{proof}
Consider an arbitrary vertex $v \in V_n \setminus U$. We aim to find a color $a \in A$ such that if we extend $\fx_U$ to $\fx_Ua^{v}$ by coloring $v$ with $a$, we preserve local consistency. Define the set $W_1 := \sigma^{MM}_n(v) \cap V^{MM}_n$. It is sufficient to check that $[\pull_{u}^{\sigma_n,MM}(\fx_Ua^{v})] \neq \emptyset$ for all $u \in W_1$, since the consistency around the remaining vertices is not affected by how we color $v$: if $u \notin V^{MM}_n$, then $\pull_{u}^{\sigma_n,MM}(\fx_Ua^{v})$ is allowed to be any pattern, and if $u \in V^{MM}_n$ but $u \notin \sigma^{MM}_n(v)$, then $v \notin \sigma^{MM}_n(u)$, since $u$ is $MM$-good and $M$ is symmetric. We can assume that $W_1 \neq \emptyset$; the empty case is trivial.

Now, consider the set $W_2 := \sigma^{MM}_n(v) \cap \sigma^M_n(W_1) \cap U$ and the pattern $w = \pull^{\sigma_n,MM}_v(\fx_{W_2})$, whose support is $F = \{g \in MM: \sigma^g_n(v) \in W_2\}$. Notice that $1_\Gamma \notin F$. We claim that $w$ is globally admissible in $X$. Indeed, if not, by the TSSM property, it would mean that there exists $g \in F$ such that $w_{F \cap Mg}$ is not globally admissible, so $[\pull^{\sigma_n,M}_{\sigma^g(v)}(\fx_{W_2})] = \emptyset$. Since $g \in F$, we have that $\sigma_n^g(v) \in W_2$. In particular, $\sigma_n^g(v) \in \sigma^M_n(W_1)$, so $\sigma_n^g(v) = \sigma_n^h(u)$ for some $u \in W_1$ and $h \in M$. Therefore, $[\pull^{\sigma_n,MM}_{u}(\fx_{W_2})] = \emptyset$, so $[\pull^{\sigma_n,MM}_{u}(\fx_{U})] = \emptyset$, which contradicts that $\fx_U$ is locally consistent, since $W_1 \subseteq  V^{MM}_n$.

Then, $w$ is globally admissible in $X$ and admits a one symbol extension at $1_\Gamma$ that preserves the global admissibility. Extend $\fx_U$ by coloring $v$ with such symbol, say $a$. Notice that $\fx_U a^{v}$ remains locally consistent. If not, it would exist $u \in W_1$ such that $[\pull_{u}^{\sigma_n,MM}(\fx_Ua^{v})] = \emptyset$. By Lemma \ref{lem:consist}, this is only possible if $[\pull_{u}^{\sigma_n,MM}(\fx_{U \cap W_2}a^{v})] = \emptyset$, but this contradicts our choice of $a$, since $[\pull_{v}^{\sigma_n,MM}(\fx_{W_2}a^{v})] \neq \emptyset$.

Iterating this procedure, we obtain an extension from $U$ to $V_n$.
\end{proof}

%

Given $\fx \in A^{V_n}$, we define its {\bf set of errors} as
$$
E(\fx) := \{v \in V_n^{MM}: \pull^{\sigma_n,MM}_v(\fx) \notin X_{MM}\}.
$$

We have the following very useful consequence of Lemma \ref{lem:extend}.

\begin{lemma}
\label{lem:error}
If $X$ satisfies the TSSM property with range $M$, then, for every $\fx \in A^{V_n}$, there exists $\fy \in X^n$ such that
$$
\fy_{V_n \setminus \sigma_n^{MM}(E(\fx))} = \fx_{V_n \setminus \sigma_n^{MM}(E(\fx))}.
$$
\end{lemma}

\begin{proof}
Consider the partial configuration $\fx_{V_n \setminus \sigma_n^{MM}(E(\fx))}$ and notice that it is locally consistent, since all the vertices $v$ in $V_n^{MM}$ such that $\sigma_n^{MM}(v) \in E(\fx)$ are contained in $\sigma_n^{MM}(E(\fx))$. By Lemma \ref{lem:extend}, it admits a locally consistent extension $\fy$ to $V_n$.
\end{proof}

In simple words, Lemma \ref{lem:error} guarantees that the ``errors'' in $\fx \in A^{V_n}$ can be ``corrected'' by replacing some portion of $\fx$ with size proportional to the error itself.

Given $M\Subset \Gamma$, we say that a potential $\phi: X \to \R$ has {\bf range $M$} if $\phi(x) = \phi(y)$ for all $x, y \in X$ with $x_M = y_M$ and that it has {\bf finite range} if it has range $M$ for some finite $M$. From now on, we will assume that $X$ has the TSSM property and $\phi$ has finite range and, without loss of generality, we will suppose that there exists $M \Subset \Gamma$ such that the range of $X$ and $\phi$ are both $M$. In addition, we can always suppose that $\phi$ is defined in all $A^\Gamma$, preserving boundedness and range. We define the {\bf $n$th derived energy} and the {\bf $n$th derived partition function} as
$$
\fEn_n(\fx) := \sum_{v \in V_n} \phi(\pull^{\sigma_n}_v(\fx))	\quad	\text{ and }	\quad	Z_n := \sum_{\fx \in X^n} \exp(\fEn_n(\fx)),
$$
respectively. In the following theorem, we prove that, provided the subshift $X$ satisfies the TSSM property, we can approximate the topological sofic pressure by means of the derived partition functions.

\begin{theorem}
\label{thm:partition}
Let $X$ be a subshift that satisfies the TSSM property and $\phi: X \to \R$ a finite range potential. Then,
$$
p_\Sigma(\Gamma \acts X ,\phi) = \limsup_n |V_n|^{-1} \log Z_{n},
$$
where $Z_n = \sum_{\fx \in X^n} \exp(\sum_{v \in V_n} \phi(\pull^{\sigma_n}_v(\fx)))$.
\end{theorem}

\begin{proof}
Given $\delta > 0$ and $E \subseteq V_n^{MM}$, we define the auxiliary sets
$$
X^{n,\delta} = \{\fx \in A^{V_n}: |E(\fx)| \leq \delta|V_n|\}	\quad	\text{ and }	\quad	X^{n,E} = \{\fx \in A^{V_n}: E(\fx) \subseteq E\},
$$
and the associated partition functions
$$
Z_{n,\delta} = \sum_{\fx \in X^{n,\delta}} \exp(\fEn_n(\fx))	\quad	\text{ and }	\quad	Z_{n,E} = \sum_{\fx \in X^{n,E}} \exp(\fEn_n(\fx)),
$$
respectively. By Lemma \ref{lem:error}, there is a natural surjection from $X^n \times A^{\sigma_n^{MM}(E)}$ to $X^{n,E}$, so
\begin{align*}
Z_{n,E}	&	= 	\sum_{\fx \in X^{n,E}} \exp(\fEn_n(\fx))	\\
		&	\leq	|A|^{|\sigma_n^{MM}(E)|}\sum_{\fy \in X^n} \exp(\fEn_n(\fy) + |\sigma_n^{MM}(E)|\|\phi\|)	\\
		&	\leq	\exp(|E||M|^2(\log|A| + \|\phi\|)) Z_n.
\end{align*}

Next, notice that
\begin{align*}
Z_{n,\delta}	&	\leq	\sum_{E \subseteq V_n: |E| \leq \delta|V_n|} Z_{n,E}									\\
			&	\leq	\sum_{E \subseteq V_n: |E| \leq \delta|V_n|} \exp(\delta|V_n||M|^2(\log|A| + \|\phi\|)) Z_n		\\
			&	\leq	{|V_n| \choose \delta|V_n|} \exp(\delta|V_n||M|^2(\log|A| + \|\phi\|)) Z_n,
\end{align*}
where we assume that $\delta|V_n|$ is an integer just to ease the notation. By Stirling's approximation for ${|V_n| \choose \delta|V_n|}$, we can write
$$
\log {|V_n| \choose \delta|V_n|} = |V_n|H(\delta, 1-\delta) + o(|V_n|),
$$
where $H(\delta, 1-\delta) = \delta \log \frac{1}{\delta}+(1-\delta) \log \frac{1}{1-\delta}$. Therefore,
$$
\limsup_n |V_n|^{-1} \log Z_n = \inf_{\delta > 0}\limsup_n |V_n|^{-1} \log Z_{n,\delta}.
$$

Now, define the auxiliary set
$$
\Map_{MM}(r, \delta, \sigma_n) := \{\mx \in X^{V_n}: |\{v \in V_n^{MM}: x_v(g) = x_{\sigma^g(v)}(1_\Gamma) \text{ for all } g \in F_r\}|	> (1-\delta)|V_n|\}.
$$
and
$$
Z_{MM}(\phi, r, \delta, \sigma_n) =  \sup\{Z_{\sigma_n}(\mathcal{Z}): \mathcal{Z} \subseteq \Map_{MM}\left(r, \delta, \sigma_n\right) \text{ is $\left(\rho_{\infty}, 1\right)$-separated}\}.
$$

Clearly, $\Map_{MM}(r, \delta, \sigma_n) \subseteq \Map(r, \delta, \sigma_n)$. Moreover, if $|V_n \setminus V^{MM}_n| \leq \delta |V_n|$, then $\Map(r, \delta, \sigma_n)  \subseteq \Map_{MM}(r, 2\delta, \sigma_n)$. Given $r\in \N$ such that $MM \subseteq F_r$ and $\delta > 0$, pick an arbitrary $(\rho_\infty,1)$-separated set $\mathcal{Z}  \subseteq \Map_{MM}(r, \delta, \sigma_n)$ and consider the map $\theta_{\sigma_n}$ restricted to $\mathcal{Z}$. Then, such map is injective and $\theta_{\sigma_n}(\mx) \in X^{n,\delta}$. In addition,
$$
|\En_n(\mx) - \fEn_n(\theta_n(\mx))|	\leq	2(|V_n \setminus V_n^{MM}| + |E(\theta_{\sigma_n}(\mx))|)\|\phi\|.
$$

In particular, if $|V_n \setminus V_n^{MM}| \leq \delta |V_n|$ and $\theta_{\sigma_n}(\mx) \in X^{n,\delta}$, then $|\En_n(\mx) - \fEn_n(\theta_{\sigma_n}(\mx))| \leq 4\delta|V_n|\|\phi\|$. Therefore,
\begin{align*}
Z_{\sigma_n}(\mathcal{Z})	&	=	\sum_{\mx \in \mathcal{Z}} \exp(\En_n(\mx))	\\
					&	\leq	\sum_{\mx \in \mathcal{Z}} \exp(\fEn_n(\theta_{\sigma_n}(\mx)) + 4\delta|V_n|\|\phi\|)	\\
					&	\leq	\exp(4\delta|V_n|\|\phi\|)\sum_{\fx \in X^{n,\delta}} \exp (\fEn_n(\fx))	\\
					&	=	\exp(4\delta|V_n|\|\phi\|)Z_{n,\delta}.
\end{align*}

Then, for sufficiently large $n$, $Z(\phi, r, \delta, \sigma_n) \leq \exp(2\delta|V_n|\|\phi\|)Z_{n,\delta}$, so $p_\Sigma(\Gamma \acts X ,\phi)  \leq \limsup_n |V_n|^{-1} \log Z_n$.

On the other hand, given $r \in \N$ and $\delta > 0$, we can pick $n_0$ so that $|V_n^{MMF_r}| \geq (1-\delta)|V_n|$ for every $n \geq n_0$. Given $\fx \in X^n$ and $v \in V_n^{MMF_r}$, notice that for every $u \in \sigma_n^{F_r}(v)$, we have that $u \in V_n^{MM}$. Since $\pull^{\sigma_n,M}_u(\fx) \in X_M$ for every $u \in \sigma_n^{F_r}(v)$, by the TSSM property, $\pull^{\sigma_n,r}_v(\fx)$ belongs to $X_{F_r}$. For each such $v$, pick $x_v \in [\pull^{\sigma,r}_v(\fx)]$ and construct $\mx$ by filling with any $x_v$ for $v \in V_n \setminus V_n^{MMF_r}$. Then, we can construct a $(\rho_\infty,1)$-separated set $\mathcal{Z}$ from $X^n$ such that $Z_n \leq Z_{\sigma_n}(\mathcal{Z})\exp(\delta|V_n|)$. Therefore, $p_\Sigma(\Gamma \acts X ,\phi) \geq \limsup_n |V_n|^{-1} \log Z_n$, and we conclude.
\end{proof}

We have the following corollary.

\begin{corollary}
\label{cor:nonnegative}
If $X$ satisfies the TSSM property and $\phi$ is a finite range potential, then $p_\Sigma(\Gamma \acts X, \phi) \neq -\infty$ for every sofic approximation $\Sigma$. Moreover, if $X$ is nontrivial, then $h_\Sigma(\Gamma \acts X) > 0$ for every sofic approximation $\Sigma$.
\end{corollary}

\begin{proof}
By Lemma \ref{lem:extend}, we have that $X^n \neq \emptyset$ and $Z_n = \sum_{\fx \in X^n} \exp(\fEn_n(\fx)) \geq |X^n|e^{-\|\phi\|}$. Let's say that $U_n \subseteq V_n$ is $M$-separated if $u,v \in U_n$ and $u \neq v$ implies that $u \notin \sigma^M(v)$ and $v \notin \sigma^M(u)$. Pick an $M$-separated set $U_n$ of vertices in $V_n^{MM}$  with maximum cardinality. Then, we can $A$-color such vertices independently to obtain locally consistent partial configurations and, by Lemma \ref{lem:extend}, we can extend each of such partial configurations to a configuration in $X^n$. Therefore, $|X^n| \geq |A|^{|U_n|}$. In addition, notice that $|U_n| \geq \frac{|V_n^{MM}|}{|MM|}$ and $\lim_n \frac{|V_n^{MM}|}{|V_n|} = 1$, so
\begin{align*}
p_\Sigma(\Gamma \acts X, \phi)	&	\geq	\limsup_n |V_n|^{-1}\log |X^n| - \|\phi\|	\\
							&	\geq	\frac{\log |A|}{|MM|} \limsup_n \frac{|V^{MM}_n|}{|V_n|} - \|\phi\|	\\
							&	=	\frac{\log |A|}{|MM|} - \|\phi\|.
\end{align*}

In particular, if $\phi \equiv 0$, we get that $h_\Sigma(\Gamma \acts X) \geq  \frac{\log |A|}{|MM|} > 0$.
\end{proof}

\section{Local weak* convergence and Gibbs measures}
\label{sec4}

We say that a sequence $\{\nu_n\}_n$, with $\nu_n \in \Prob(A^{V_n})$, {\bf locally weak* converges} to $\nu \in \Prob(A^\Gamma)$ if for every $\epsilon > 0$ and for every $r \in \N$,
$$
\lim_{n \rightarrow \infty} \frac{1}{|V_n|} |\left\{v \in V_n: \left\|\left(\pull_{v}^{\sigma_n, r}\right)_{*} \nu_n\vert_{F_r} - \nu\vert_{F_r}\right\|_{\mathrm{TV}} > \epsilon\right\}| = 0,
$$
where $\|\nu-\mu\|_{\mathrm{TV}}$ is the \emph{total variation distance}, $\nu\vert_F$ is the marginal of $\nu$ on $A^F$, and $\{F_r\}_r$ is the fixed exhaustion. We also consider the metric
$$
d(\mu,\nu) := \sum_{r} \frac{1}{2^r} \|\mu\vert_{F_r} - \nu\vert_{F_r}\|_{\mathrm{TV}}
$$
 in $\Prob(A^\Gamma)$, which is compatible with the weak* topology, and define the map $\pull_\uparrow^{\sigma_n}: \Prob(A^{V_n}) \to \Prob(A^\Gamma)$ by the formula
$$
\pull^{\sigma_n}_\uparrow(\nu_n) := \frac{1}{|V_n|} \sum_{v \in V} (\pull_{v}^{\sigma_n})_{*}(\nu_n),
$$
where we add the index $n$ in order to emphasize the space where $\nu_n$ belongs. Given $r \in \N$, $\epsilon > 0$, and $\nu \in \Prob(A^\Gamma)$, we also define the sets
$$
O(r,\epsilon,\nu) := \{\nu' \in \Prob(A^\Gamma): \left\|\nu'\vert_{F_r} - \nu\vert_{F_r}\right\|_{\mathrm{TV}} \leq \epsilon\},
$$
which are a basis for the weak* topology. In particular, it follows that if $\{\nu_n\}_n$ locally weak* converges to $\nu$, then $\lim_n d(\pull^{\sigma_n}_\uparrow(\nu_n), \nu) = 0$, i.e., $\pull^{\sigma_n}_\uparrow(\nu_n)$ weak* converges to $\nu$. In addition, and for convenience, given $\delta > 0$ and $\nu \in \Prob(A^\Gamma)$, we also define the open ball with radius $\delta$ centered at $\nu$ as
$$
B_d(\nu,\delta) := \{\nu' \in \Prob(A^\Gamma): d(\nu',\nu) < \delta\}.
$$



\subsection{Empirical distributions}

Given $\fx \in A^{V_n}$, we denote by $P^{\sigma_n}_{\fx}$ the {\bf empirical distribution} $\pull_*^{\sigma_n}(\delta_{\fx}) \in \Prob(A^\Gamma)$, i.e.,
$$
P_{\fx}^{\sigma_n} = \frac{1}{|V_n|} \sum_{v \in V_n} \delta_{\pull_{v}^{\sigma_n}(\fx)}.
$$

In this symbolic setting, it is known (see \cite{austin2016additivity}) that the sofic entropy of $\nu \in \Prob(A^\Gamma, \Gamma)$ can be obtained by the formula
$$
h_{\Sigma}(\Gamma \acts (A^{\Gamma}, \nu)) = \inf_{\epsilon > 0} \inf_{r \in \N} \limsup _{n \rightarrow \infty} \frac{1}{n} \log |\Omega_\nu(r,\epsilon,\sigma_{n})|,
$$
where $\Omega_{\nu}(r, \epsilon, \sigma_n) := \left\{\fx \in A^{V_n}:\left\| P_{\fx}^{\sigma_n}\vert_{F_r} - \nu_{F_r}\right\|_{\mathrm{TV}} < \epsilon\right\}$. It is not difficult to see that $h_{\Sigma}(\Gamma \acts (A^{\Gamma}, \nu)) \neq -\infty$ if and only if there exists a sequence $\{\fx_n\}_n$, with $\fx_n \in A^{V_n}$, such that $P_{\fx_n}^{\sigma_n}$ weak* converges to $\nu$. Moreover, as we establish in the following lemma, if $\nu \in \Prob(X,\Gamma)$ and $X$ satisfies the TSSM property, then we can choose the sequence in $X^n$ instead of just $A^{V_n}$.

\begin{lemma}
\label{lem:Xn}
If $\nu \in \Prob(X,\Gamma)$, $X$ satisfies the TSSM property, and $h_{\Sigma}(\Gamma \acts (A^{\Gamma}, \nu)) \neq -\infty$, then there exists a sequence $\fy_n \in X^n$ such that $P_{\fy_n}^{\sigma_n}$ weak* converges to $\nu$.
\end{lemma}
\begin{proof}
If $h_{\Sigma}(\Gamma \acts (A^{\Gamma}, \nu)) \neq -\infty$, there exists a sequence $\fx_n \in A^{V_n}$ such that $P_{\fx_n}^{\sigma_n}$ weak* converges to $\nu$. Now, for each $\fx_n$, consider $\fy_n \in X^n$ with $(\fy_n)_{V_n \setminus \sigma_n^{MM}(E(\fx_n))} = (\fx_n)_{V_n \setminus \sigma_n^{MM}(E(\fx_n))}$,
which is provided by Lemma \ref{lem:error}. 

Given $\epsilon > 0$ and $r \in \N$, we will prove that for sufficiently large $n$, $\|P_{\fy_n}^{\sigma_n}\vert_{F_r} - \nu\vert_{F_r}\|_{\mathrm{TV}} \leq \epsilon$. Without loss of generality, suppose that $MM \subseteq F_r$, where $M$ denotes some range for the TSSM property. First, observe that
\begin{align*}
\|P_{\fx_n}^{\sigma_n}\vert_{F_r} - P_{\fy_n}^{\sigma_n}\vert_{F_r}\|_{\mathrm{TV}}	 &	\leq	\frac{1}{|V_n|} \sum_{v \in V_n} \| \delta_{\pull_{v}^{\sigma_n}(\fx_n)} \vert_{F_r} - \delta_{\pull_{v}^{\sigma_n}(\fy_n)} \vert_{F_r} \|_{\mathrm{TV}}	\\
	&	=	\frac{1}{|V_n|} \sum_{v \in V_n} \frac{1}{2} \sum_{w \in A^{F_r}} | \delta_{\pull_{v}^{\sigma_n}(\fx_n)}([w]) - \delta_{\pull_{v}^{\sigma_n}(\fy_n)}([w]) |	\\
	&	=	\frac{1}{|V_n|} |\{v \in V_n: \pull_{v}^{\sigma_n,r}(\fx_n)  \neq \pull_{v}^{\sigma_n,r}(\fy_n)\}|	\\
	&	\leq	\frac{1}{|V_n|} (|V_n \setminus V^{MMF_r}_n| \\
	&		\qquad + |\{v \in V^{MMF_r}_n: \sigma_n^{F_r}(v) \cap \sigma_n^{MM}(E(\fx_n)) \neq \emptyset\}|)	\\
	&	\leq	\frac{1}{|V_n|} (|V_n \setminus V^{MMF_r}_n| + |M|^2|F_r||E(\fx_n)|).
\end{align*}

On the other hand, since $\nu$ is supported on $X$,
\begin{align*}
\|P_{\fx_n}^{\sigma_n}\vert_{MM} - \nu\vert_{MM}\|_{\mathrm{TV}}	&	=		\frac{1}{2}\sum_{w \in A^{MM}} |P_{\fx_n}^{\sigma_n}([w]) - \nu([w])|	\\
				&	\geq		\frac{1}{2}\sum_{w \in A^{MM} \setminus X_{MM}} P_{\fx_n}^{\sigma_n}([w])	\\
				&	=		\frac{1}{2}\sum_{w \in A^{MM} \setminus X_{MM}} \frac{1}{|V_n|}\sum_{v \in V_n} \delta_{\pull_{v}^{\sigma_n}(\fx_n)}([w])	\\
				&	=		\frac{1}{2|V_n|}\sum_{v \in V_n} \sum_{w \in A^{MM} \setminus X_{MM}}\delta_{\pull_{v}^{\sigma_n}(\fx_n)}([w])	\\
				&	=		\frac{1}{2|V_n|} |\{v \in V_n: \pull_{v}^{\sigma_n,{MM}}(\fx_n) \notin X_{MM}\}|.
\end{align*}

In particular, $|E(\fx_n)| \leq 2|V_n|\|P_{\fx_n}^{\sigma_n}\vert_{F_{MM}} - \nu\vert_{F_{MM}}\|_{\mathrm{TV}}$. Therefore,
\begin{equation}
\label{eq1}
\|P_{\fx_n}^{\sigma_n}\vert_{F_r} - P_{\fy_n}^{\sigma_n}\vert_{F_r}\|_{\mathrm{TV}}	\leq	\frac{|V_n \setminus V^{MMF_r}_n|}{|V_n|} + 2|M|^2|F_r|\|P_{\fx_n}^{\sigma_n}\vert_{F_{MM}} - \nu\vert_{F_{MM}}\|_{\mathrm{TV}}.
\end{equation}

Now, given $\epsilon > 0$, pick $n_0 \in \N$ so that $|V_n \setminus V^{MMF_r}_n| \leq \frac{\epsilon}{3}|V_n|$, $\|P_{\fx_n}^{\sigma_n}\vert_{F_r} - \nu\vert_{F_r}\|_{\mathrm{TV}} \leq \frac{\epsilon}{3}$, and $\|P_{\fx_n}^{\sigma_n}\vert_{F_{MM}} - \nu\vert_{F_{MM}}\|_{\mathrm{TV}} \leq \frac{\epsilon}{6|M|^2|F_r|}$ for every $n \geq n_0$, which is provided by the weak* convergence. Then,
\begin{align*}
\|P_{\fy_n}^{\sigma_n}\vert_{F_r} - \nu\vert_{F_r}\|_{\mathrm{TV}}	&	\leq	\|P_{\fy_n}^{\sigma_n}\vert_{F_r} - P_{\fx_n}^{\sigma_n}\vert_{F_r}\|_{\mathrm{TV}}	+ \|P_{\fx_n}^{\sigma_n}\vert_{F_r} - \nu\vert_{F_r}\|_{\mathrm{TV}}	\\
													&	\leq	\left(\frac{|V_n \setminus V^{MMF_r}_n|}{|V_n|} + 2|M|^2|F_r|\|P_{\fx_n}^{\sigma_n}\vert_{F_{MM}} - \nu\vert_{F_{MM}}\|_{\mathrm{TV}}\right) + \frac{\epsilon}{3}	\\
													&	\leq	\frac{\epsilon}{3} + \frac{\epsilon}{3} + \frac{\epsilon}{3} = \epsilon,
\end{align*}
and we conclude.
\end{proof}

In the next results we establish some useful invariance properties.

\begin{lemma}[{\cite[Lemma 3.2]{austin2016additivity}}]
\label{lemma32}
If $F \Subset \Gamma$ and $g \in \Gamma$, then
$$
\sup _{\fx \in A^{V_{n}}}\left\|P_{\fx}^{\sigma_{n}}\vert_{F} - T_{*}^{g} P_{\fx}^{\sigma_{n}}\vert_{F}\right\|_{\mathrm{TV}} \to 0 \quad \text { as } n \to \infty.
$$
\end{lemma}

\begin{proof}
Lemma \ref{lemma31} gives that
$$
\left.\left(T^{g} \pull_{v}^{\sigma_{n}}(\fx)\right)\right|_{F}=\left.\pull_{\sigma_{n}^{g}(v)}^{\sigma_{n}}(\fx)\right|_{F} \quad \text { w.h.p. in } v,
$$
and therefore,
\begin{align*}
\left.P_{\fx}^{\sigma_{n}}\right\vert_{F} - \left.T_{*}^{g} P_{\fx}^{\sigma_{n}}\right\vert_{F}	&	=	\frac{1}{\left|V_{n}\right|} \sum_{v \in V_{n}} \left(\delta_{\pull_{v}^{\sigma_{n}}(\fx)\vert_{F}}-\delta_{\left(T^{g} \pull_{v}^{\sigma_{n}}(\fx)\right)\vert_{F}}\right)	\\
																	&	=	\frac{1}{\left|V_{n}\right|} \sum_{v \in V_{n}} \left(\delta_{\pull_v^{\sigma_{n}}(\fx)\vert_F} - \delta_{\pull_{\sigma_n^g(v)}^{\sigma_n}(\fx)\vert_F}\right)+o(1).
\end{align*}

By observing that the last sum vanishes because $\sigma_{n}^{g}$ is a permutation of $V_{n}$, we conclude.
\end{proof}

Lemma \ref{lemma32} has the following consequence.

\begin{corollary}
\label{lem:inv}
Let $F \Subset \Gamma$ and $g \in \Gamma$. Then,
$$
\sup_{\nu_n \in \Prob(A^{V_n})}\left\|\pull^{\sigma_n}_\uparrow(\nu_n)\vert_{F} - T_{*}^{g} \pull^{\sigma_n}_\uparrow(\nu_n)\vert_{F}\right\|_{\mathrm{TV}} \to 0 \quad \text { as } n \to \infty.
$$
\end{corollary}

\begin{proof}
For any $\nu_n \in \Prob(A^{V_n})$, there exist weights $0 \leq \alpha_{\fx} \leq 1$ with $\sum_{\fx \in X^n} \alpha_{\fx} = 1$, such that $\nu_n = \sum_{\fx \in X^n} \alpha_{\fx} \delta_{\fx}$. Therefore,
\begin{align*}
\left\|\pull^{\sigma_n}_\uparrow(\nu_n)\vert_{F} - T_{*}^{g} \pull^{\sigma_n}_\uparrow(\nu_n)\vert_{F}\right\|_{\mathrm{TV}}	&	\leq \sum_{\fx \in A^{V_n}} \alpha_{\fx} \left\|P_{\fx}^{\sigma_{n}}\vert_{F} - T_{*}^{g} P_{\fx}^{\sigma_{n}}\vert_{F}\right\|_{\mathrm{TV}}	\\
																				&	\leq	\sup _{\fx \in A^{V_n}}\left\|P_{\fx}^{\sigma_{n}}\vert_{F} - T_{*}^{g} P_{\fx}^{\sigma_{n}}\vert_{F}\right\|_{\mathrm{TV}}.
\end{align*}

Considering this, and since $\nu_n$ is arbitrary, we have that
$$
\sup_{\nu_n \in \Prob(A^{V_n})}\left\|\pull^{\sigma_n}_\uparrow(\nu_n)\vert_{F} - T_{*}^{g} \pull^{\sigma_n}_\uparrow(\nu_n)\vert_{F}\right\|_{\mathrm{TV}} \leq \sup_{\fx \in  A^{V_n}}\left\|P_{\fx}^{\sigma_{n}}\vert_{F} - T_{*}^{g} P_{\fx}^{\sigma_{n}}\vert_{F}\right\|_{\mathrm{TV}} \to 0
$$
as $n \to \infty$, due to Lemma \ref{lemma32}.
\end{proof}

\begin{lemma}
Given $\nu \in \Prob(A^\Gamma)$ and $\{\nu_n\}_n$ with $\nu_n \in \Prob(A^{V_n})$, if $\lim_n d(\pull^{\sigma_n}_\uparrow(\nu_n), \nu) = 0$, then $\nu$ is $\Gamma$-invariant.
\end{lemma}

\begin{proof}
It suffices to prove that $\nu([x_F]) = T^g_*\nu([x_{F}])$ for any $g \in \Gamma$, $x \in X$, and $F \Subset \Gamma$. By the triangular inequality, 
\begin{align*}
|\nu([x_{F}]) - T^g_*\nu([x_F])|	&	\leq	|\nu([x_{F}]) - \pull^{\sigma_n}_\uparrow(\nu_n)([x_{F}])| + |\pull^{\sigma_n}_\uparrow(\nu_n)([x_{F}]) - T^{g}_* \pull^{\sigma_n}_\uparrow(\nu_n)([x_{F}])|	\\
						&	\qquad + |T^{g}_* \pull^{\sigma_n}_\uparrow(\nu_n)([x_{F}]) - T^g_*\nu([x_{F}])|.
\end{align*}

Given $\epsilon > 0$, by the local weak* convergence, for $r \in \N$ such that $F_r \supseteq F \cup Fg$ and sufficiently large $n$, we have that
\begin{align*}
|\nu([x_{F}]) - \pull^{\sigma_n}_\uparrow(\nu_n)([x_{F}])|  \leq \left\| \nu\vert_{F_r} - \pull^{\sigma_n}_\uparrow(\nu_n)\vert_{F_r}\right\|_{\mathrm{TV}} < \frac{\epsilon}{3}
\end{align*}
and
\begin{align*}
|T^{g}_* \pull^{\sigma_n}(\nu_n)([x_{F}]) - T^g_*\nu([x_{F}])|	&	=	|\pull^{\sigma_n}_\uparrow(\nu_n)([(T^{g^{-1}}x)_{Fg}]) - \nu([(T^{g^{-1}}x)_{Fg}])|	\\
												&	\leq	\left\|\pull^{\sigma_n}_\uparrow(\nu_n)\vert_{F_r} - \nu\vert_{F_r}\right\|_{\mathrm{TV}} < \frac{\epsilon}{3}
\end{align*}

In addition, by Lemma \ref{lem:inv},
\begin{align*}
|\pull^{\sigma_n}_\uparrow(\nu_n)([x_F]) - T^{g}_* \pull^{\sigma_n}_\uparrow(\nu_n)([x_F])|	\leq	\left\|\pull^{\sigma_n}_\uparrow(\nu_n)\vert_{F_r} - T^{g}_* \pull^{\sigma_n}_\uparrow(\nu_n)\vert_{F_r}\right\|_{\mathrm{TV}} < \frac{\epsilon}{3},
\end{align*}
and since $\epsilon$ was arbitrary, by combining these three inequalities, we conclude.

\end{proof}


Given a subshift $X$ and a potential $\phi$, a relevant family of measures in $\Prob(X)$ are the so-called \emph{Gibbs measures}.

\subsection{Gibbs measures}

Given $F \Subset \Gamma$, we define the {\bf $F$-sum} as $\phi_F(x) := \sum_{g \in F} \phi(T^g x)$ and the {\bf $(F,x)$-partition function} as $Z_F(y) = \sum_{x \in [y_{F^c}]} \exp(\phi_F(x))$. Considering this, and given $y \in X$, we define the {\bf $(F,y)$-specification} as
$$
\gamma_F(x \vert y) = 1_{[y_{F^c}]}(x) \frac{\exp(\phi_F(x))}{Z_F(y)},
$$
for every $x \in X$. We call the collection $\gamma = \{\gamma_F(\cdot \vert y)\}_{F,y}$ the {\bf $(X,\phi)$-specification}. Notice that each element $\gamma_F(\cdot \vert y)$ is a probability distribution on $X$. We say that a measure $\mu \in \Prob(A^\Gamma)$ is a {\bf Gibbs measure for $\phi$ on $X$} if, for all $F \Subset \Gamma$,
$$
\mu([x_F] \vert \mathcal{B}_{\Gamma \setminus F})(y) = \gamma_F(x \vert y) \quad \mu(y)\text{-a.s.},
$$
where $\mathcal{B}_{\Gamma \setminus F}$ is the $\sigma$-algebra generated by $\{[x_E]: E \Subset \Gamma \setminus F\}$ and $\nu(A | \mathcal{F})$ denotes the conditional expectation of the indicator function $1_{A}$ for $A \in \mathcal{B}_F$, with respect to $\nu \in \Prob(A^\Gamma)$ and the sub-$\sigma$-algebra $\mathcal{F}$. Since we are assuming that $\phi$ has finite range $M$, it turns out that $\gamma$ is {\bf Markovian}, in the sense that for all $F \Subset \Gamma$, we have that
$$
\gamma_F(\cdot \vert y_1) = \gamma_F(\cdot \vert y_2),
$$
for all $y_1,y_2 \in X$ such that $y_1 \vert_{MF} = y_2 \vert_{MF}$. We define the {\bf $M$-boundary} of a set $F$ as the set $MF \setminus F$ and denote it by $\partial_M F$. In other words, $\gamma_F(\cdot \vert y)$ only depends on $y_{\partial_M F}$. Notice that the support of $\mu$ is always contained in $X$ and it is not difficult to prove that under mixing assumptions like the TSSM property or weaker, then $\mu$ is fully supported (e.g., see \cite{1-ruelle}).

\subsection{Derived models}

We define Gibbs measures in $X^n$ analogously. Given $U \subseteq V_n$ and $\fx,\fy \in X^n$, set
$$
\fEn_{n,U}(\fx) := \sum_{v \in U \cap V_n} \phi(\pull^{\sigma_n}_v(\fx)) \quad \text{ and }	\quad	Z_{n,U}(\fy) := \sum_{\fx \in [\fy_{U^c}]} \exp(\fEn_{n,U}(\fx)),
$$
where we extend the cylinder notation to $X^n$ in the natural way. We define the {\bf $(U,\fx)$-specification} as
$$
\gamma^n_U(\fx \vert \fy) = 1_{[\fy_{U^c}]}(\fx) \frac{\exp(\fEn_{n,U}(\fx))}{Z_{n,U}(\fy)},
$$
for every $\fy \in X^n$ and $\fx \in A^{V_n}$. We call the collection $\gamma^n = \{\gamma^n_U(\cdot \vert \fy)\}_{U,\fy}$ the {\bf $n$th derived $(X^n,\phi)$-specification}. Notice that each element $\gamma^n_U(\cdot \vert \fy)$ is a probability distribution on $A^{V_n}$. Since $V_n$ is finite, we can define Gibbs measures without considering conditional expectations. Then, the {\bf $n$th derived Gibbs measure for $\phi$ on $X^n$} is defined as
$$
\mu_n(\fx) = 1_{X^n}(\fx)\frac{\exp(\fEn_n(\fx))}{Z_n},
$$
with $Z_n = Z_{n,\emptyset} = \sum_{\fx \in X^n} \phi(\pull^{\sigma}_v(\fx))$. Notice that the support of $\mu_n$ is $X^n$. We have the following simple but relevant preliminary results.


\begin{lemma}[{\cite[Lemma 3.1]{austin2016additivity}}]
\label{lemma31}
If $F \Subset \Gamma$ and $g \in \Gamma$, then, w.h.p. in $v \in V_n$,
$$
\pull_{\sigma_{n}^{g}(v)}^{\sigma_{n},F} = T^g \pull_{v}^{\sigma_{n},Fg},
$$
where we identify elements of $A^{F g}$ with elements of $A^{F}$ in the obvious way.
\end{lemma}

\begin{proof}
It suffices to prove this when $F$ is an arbitrary singleton, say $\{h\}$. Then it holds w.h.p. in $v \in V_n$ that $\sigma_{n}^{h}(\sigma_{n}^{g}(v)) = \sigma_{n}^{h g}(v)$. If $v$ satisfies this, and $\fx \in X^n$, then
$$
(\pull_{\sigma_{n}^{g}(v)}^{\sigma_{n}}(\fx))(h) = x_{\sigma_{n}^{h}(\sigma_{n}^{g}(v))} = x_{\sigma_{n}^{h g}(v)} = (\pull_{v}^{\sigma_{n}}(\fx))(hg) = T^g(\pull_{v}^{\sigma_{n}}(\fx))(h),
$$
and we conclude.
\end{proof}

\begin{lemma}
\label{lem:spec}
Given $F \Subset \Gamma$, w.h.p. in $v \in V_n$, for every $\fx,\fy \in X^n$,
$$
\gamma^{n}_{\sigma_n^F(v)}(\fx \vert \fy) = \gamma_F(\pull_{v}^{\sigma_n}(\fx) \vert \pull_{v}^{\sigma_n}(\fy)).
$$
\end{lemma}

\begin{proof}
Notice that, w.h.p. in $v \in V_n$, for every $\fx,\fy \in X^n$,
$$
\fx_{\sigma_n^F(v)}\fy_{\sigma_n^{MF \setminus F}}(v) \in X^n_{\sigma_n^{MF}(v)} \iff \pull_{v}^{\sigma_n}(\fx)_F \pull_{v}^{\sigma_n}(\fy)_{MF \setminus F} \in X_{MF}.
$$

In addition, again w.h.p. in $v \in V_n$, by Lemma \ref{lemma31},
\begin{align*}
\fEn_{n,\sigma_n^F(v)}(\fx)	=	\sum_{u \in \sigma_n^F(v)} \phi(\pull^{\sigma_n}_u(\fx))	=	\sum_{g \in F} \phi(\pull^{\sigma_n}_{\sigma_n^g(v)}(\fx))	=	\sum_{g \in F} \phi(T^g\pull^{\sigma_n}_{v}(\fx))	=	\phi_F(\pull^{\sigma_n}_{v}(\fx)).
\end{align*}

Therefore, w.h.p. in $v \in V_n$, $Z_{n,\sigma_n^F(v)}(\fy) = Z_F(\pull_{v}^{\sigma_n}(\fy))$, and we conclude.
\end{proof}

%

\subsection{Uniqueness and local weak* convergence}

An important question in statistical phys\-ics is whether an $(X,\phi)$-specification has a unique or multiple Gibbs measures. It is well-known that in the first case and under our assumptions on the support ---namely, the TSSM property---, uniqueness is characterized by a certain decay of correlation at the level of the specification (e.g., see \cite{3-weitz}). More specifically, if there is a unique Gibbs measure, then, for every $\epsilon > 0$, $r \in \N$, and $x, y_1,y_2 \in X$,
$$
|\gamma_{F_{r'}}([x_{F_r}] \vert y_1) - \gamma_{F_{r'}}([x_{F_r}] \vert y_2)| \leq \epsilon,
$$
for a sufficiently large $r' > r$, where $\gamma_F([x_F] \vert y)$ denotes the natural marginalization given by
$$
\gamma_F([x_F] \vert y) = \sum_{x' \in [x_F]} \gamma_F(x' \vert y).
$$

We have the following relationship between uniqueness and convergence in the local weak* sense.

\begin{proposition}
\label{lwconv}
Suppose that there is a unique Gibbs measure $\mu$ for $\phi$. Then, the sequence of derived Gibbs measures $\{\mu_n\}_n$  converges to $\mu$ in the local weak* sense.
\end{proposition}

\begin{proof}
Pick an arbitrary $r \in \N$ and $\epsilon > 0$. By uniqueness, there exists $r' > r$ such that for every $w \in X_{F_r}$ and for every $y_1,y_2 \in X$,
$$
|\gamma_{F_{r'}}([w] \vert y_1) - \gamma_{F_{r'}}([w] \vert y_1)| \leq \epsilon.
$$

In particular, for every $y \in X$, since $\mu([w])$ is an average of terms of the form $\gamma_{F_{r'}}([w] \vert y)$, we have that
$$
|\mu([w]) - \gamma_{F_{r'}}([w] \vert y)| \leq \frac{\epsilon}{3}.
$$

Similarly, by Lemma \ref{lem:spec}, w.h.p. in $v \in V_n$, for $\fx,\fy \in X^n$,
$$
\gamma^n_{\sigma_n^{F_{r'}}(v)}(\fx \vert \fy) = \gamma_{F_{r'}}(\pull_{v}^{\sigma_n}(\fx) \vert \pull_{v}^{\sigma_n}(\fy)),
$$
so
$$
\gamma^n_{\sigma_n^{F_{r'}}(v)}([\fx_{\sigma^{F_r}_n(v)}] \vert \fy) = \gamma_{F_{r'}}([\pull_{v}^{\sigma_n,r}(\fx)] \vert \pull_{v}^{\sigma_n}(\fy)),
$$
and
$$
|\mu_n([\fx_{\sigma^{F_r}_n(v)}]) - \gamma_{F_{r'}}([\pull_{v}^{\sigma_n,r}(\fx)]\vert \pull_{v}^{\sigma_n}(\fy))| \leq \frac{\epsilon}{3},
$$
or, in other words,
$$
|(\pull_{v}^{\sigma_n})_*\mu_n([\pull_{v}^{\sigma_n,r}(\fx)]) - \gamma_{F_{r'}}([\pull_{v}^{\sigma_n,r}(\fx)] \vert \pull_{v}^{\sigma_n}(\fy))| \leq \frac{\epsilon}{3}.
$$

Then, for any $w \in X_{F_r}$ and $\fx \in X^n$ such that $\pull^{\sigma_n,r}_v(\fx) = w$, we obtain that
\begin{align*}	
|\mu([w]) - (\pull_v^{\sigma_n,r})_*\mu_n([w])|	&	\leq |\gamma_{F_{r'}}([w] \vert y) - \gamma_{F_{r'}}([\pull_{v}^{\sigma_n,r}(\fx)] \vert \pull_{v}^{\sigma_n}(\fy))|  + \frac{2\epsilon}{3}	\\
							&	=		|\gamma_{F_{r'}}([w] \vert y) - \gamma_{F_{r'}}([w] \vert \pull_{v}^{\sigma_n}(\fy))| + \frac{2\epsilon}{3}\\
							&	\leq		\epsilon,
\end{align*}
w.h.p. in $v \in V_n$, and since $\epsilon$ was arbitrary, we conclude.

\end{proof}

\section{Sofic entropy of Gibbs measures}
\label{sec5}


Consider the map $\pull^{\sigma_n}_\downarrow: \mathcal{C}(A^\Gamma) \to \mathcal{C}(A^{V_n})$ given by
$$
(\pull^{\sigma_n}_\downarrow(\phi))(\fx) :=  \frac{1}{|V_n|} \sum_{v \in V_n} \phi(\pull_{v}^{\sigma_n}(\fx)).
$$

Notice that the map $\pull^{\sigma_n}_\uparrow$ is conjugate to $\pull^{\sigma_n}_\downarrow$, i.e., for all $\nu_n \in \Prob(A^{V_n})$,
$$
\int{\phi(x)}d(\pull^{\sigma_n}_\uparrow(\nu_n))(x) = \int{(\pull^{\sigma_n}_\downarrow(\phi))(\fx)}d\nu_n(\fx).
$$

Indeed,
\begin{align*}
\int{\phi(x)}d(\pull^{\sigma_n}_\uparrow(\nu_n))(x)	&	=	 \frac{1}{|V_n|} \sum_{v \in V_n} \int_{X}{\phi(x)} d((\pull_{v}^{\sigma_n})_{*}(\nu_n))(x)	\\
								&	=	 \frac{1}{|V_n|} \sum_{v \in V_n} \sum_{\fx \in A^{V_n}} {\phi(\pull_{v}^{\sigma_n}(\fx))}\nu_n(\{\fx\})\\
								&	=	\sum_{\fx \in A^{V_n}}{(\pull^{\sigma_n}_\downarrow(\phi))(\fx)}\nu_n(\{\fx\})	\\
								&	=	\int{(\pull^{\sigma_n}_\downarrow(\phi))(\fx)}d\nu_n(\fx).
\end{align*}

Now, given $\nu$ in $\Prob(X,\Gamma)$, a positive integer $n$, and $\delta>0$, we define
$$
\hat{h}_{\Sigma, n, \delta}(\nu) := |V_{n}|^{-1} \sup \left\{H(\nu_n) | \nu_n \in \Appr_{\Sigma,n,\delta}(\nu)\right\},
$$
where $\Appr_{\Sigma, n, \delta}(\nu) := \{\nu_n \in \Prob(A^{V_n}): d(\pull^{\sigma_n}_\uparrow(\nu_n),\nu) < \delta\}$ and $H(\nu_n)$ is the Shannon entropy of $\nu_n$. If $\Appr_{\Sigma, n, \delta}(\nu)$ is empty, we set $\hat{h}_{\Sigma, n,\delta}(\nu)=-\infty$. We also consider
$$\hat{h}_{\Sigma, \delta}(\nu) := \limsup_{n \to \infty} \hat{h}_{n, \delta}(\nu)	\quad \text{ and }	\quad \hat{h}_{\Sigma}(\nu) := \inf_{\delta > 0} \hat{h}_{\delta}(\nu).
$$

The latter quantity is sometimes called {\bf modified sofic entropy} (see \cite{alpeev2015entropy}). It is always the case that $\hat{h}_{\Sigma}(\nu) \geq h_{\Sigma}(\Gamma \acts (X,\nu))$ and, if $\nu$ is ergodic, we have that $\hat{h}_{\Sigma}(\nu) = h_{\Sigma}(\Gamma \acts (X,\nu))$ (e.g., see \cite[Section 4]{bowen2011entropy}). The following results follow \cite{alpeev2015entropy} but considering the constrained case, a generalization that involves a more delicate control of entropy.
%

\begin{lemma}
\label{estimate}
For every $\nu_n \in \Prob(X^n)$,
$$
\frac{H(\nu_n)}{|V_n|} + \int_{X} \phi(x) d(\pull^{\sigma_n}_\uparrow(\nu_n))(x) \leq \frac{H(\mu_n)}{|V_n|} + \int_{X} \phi(x) d(\pull^{\sigma_n}_\uparrow(\mu_n))(x),
$$
where $\{\mu_n\}_n$ denotes the sequence of derived Gibbs measures.
\end{lemma}

\begin{proof}
Notice that, since
$$
\int_{X} \phi(x) d\left(\pull_*^{\sigma_n}(\nu_n)\right)(x) = \frac{1}{|V_{n}|} \sum_{v \in V_n} \sum_{\fx \in X^n} {\phi(\pull_{v}^{\sigma_n}(\fx))}\nu_n(\{\fx\}),
$$
we have that
\begin{eqnarray*}
&	&H(\nu_n) + |V_n| \int_{X} \phi(x) d\left(\pull_*^{\sigma_n}(\nu_n)\right)(x)	\\
&=	& |V_n| \sum_{\fx \in X^n} \left( -\nu_n(\{\fx\})\log \nu_n(\{\fx\}) + \sum_{v \in V_n} {\phi(\pull_{v}^{\sigma_n}(\fx))}\nu_n(\{\fx{}\})\right).
\end{eqnarray*}

Let's study
$$
\sum_{\fx \in X^n} \left( -\nu_n(\{\fx\})\log \nu_n(\{\fx\}) + \sum_{v \in V_n} {\phi(\pull_{v}^{\sigma_n}(\fx))}\nu_n(\{\fx\})\right).
$$
among all measures $\nu_n \in \Prob(X^n)$. Since $\sum_{\fx \in X^n} \nu_n(\{\fx\}) = 1$ and $\nu_n(\{\fx\}) \geq 0$, we can use the Lagrange multipliers method and optimize the functional
$$
f(\{p_{\fx}\}_\fx,\lambda) = \sum_{\fx \in X^n} \left( -p_{\fx} \log p_{\fx} + \sum_{v \in V_n} {\phi(\pull_{v}^{\sigma_n}(\fx))}p_{\fx}\right) - \lambda \cdot \left(\sum_{\fx \in X^n} p_{\fx} - 1\right).
$$

Notice that
$$
\frac{\partial f}{\partial p_{\fx}} =  -\log p_{\fx} - 1 + \sum_{v \in V_n} {\phi(\pull_{v}^{\sigma_n}(\fx))} - \lambda = 0.
$$

Therefore,
$$
p_{\fx} = \frac{\exp(\sum_{v \in V_n} {\phi(\pull_{v}^{\sigma_n}(\fx))})}{\exp(1+\lambda)}.
$$

Since $\sum_{\fx \in X^n} p_{\fx}  = 1$, we conclude that
$$
p_{\fx} = \frac{\exp(\fEn_n(\fx))}{Z_n} = \mu_n(\fx).
$$
\end{proof}

The following proposition is only relevant in the constrained case and it shows how the TSSM property is a key assumption in our work to construct measures in $\Prob(X^n)$ that approximate well in the local weak* sense any given measure in $\Prob(X)$.

\begin{proposition}
\label{prop:supp}
Suppose that $X$ satisfies the TSSM property. Then, for every $\delta > 0$ and for every $\nu \in \Prob(X)$, there exists $0 < \delta' < \delta$ and $n_0 \in \N$ such that for every $n \geq n_0$ and every $\nu_n \in \Prob(A^{V_n})$ with $d(\pull^{\sigma_n}_\uparrow(\nu_n), \nu) \leq \delta'$, there exists $\tilde{\nu}_n \in \Prob(X^n)$ such that $d(\pull^{\sigma_n}_\uparrow(\tilde{\nu}_n), \nu) \leq \delta$ and $|V_n|^{-1}|H(\nu_n) - H(\tilde{\nu}_n)| \leq f(\delta)$, where $f(\delta) \to 0$ as $\delta \to 0$.
\end{proposition}

\begin{proof}
Given $\delta > 0$ and $\nu \in \Prob(X)$, let $r_0 \in \N$ and $0 < \epsilon_0 < \delta$ be such that
$$
MM \subseteq F_{r_0} \quad	\text{ and }	\quad	O(r_0,\epsilon_0,\nu) \subseteq B_d(\nu,\delta).
$$

Next, take $n_0 \in \N$ such that $|V_n \setminus V_n^{MMF_{r_0}}| \leq \frac{\epsilon_0}{4} |V_n|$ for all $n \geq n_0$ and consider $\delta' > 0$ such that 
$$
B_d(\nu,\delta') \subseteq O\left(r_0,\frac{\epsilon_0}{8|M|^2|F_{r_0}|},\nu\right).
$$ 

Given $\nu_n \in \Prob(A^{V_n})$, there exist $\{\fx^i_n\}_{i=1}^{k_n}$ in $A^{V_n}$ and weights $\{\alpha^n_i\}_{i=1}^{k_n}$ such that $\nu_n = \sum_{i=1}^{k_n} \alpha^n_i \delta_{\fx^i_n}$, with $0 \leq \alpha^n_i \leq 1$ and $ \sum_{i=1}^{k_n} \alpha^n_i = 1$. Then, it follows that, $\pull^{\sigma_n}_\uparrow(\nu_n) = \sum_{i=1}^{k_n} \alpha^n_i P^{\sigma_n}_{\fx_n^i}$. Notice that, due to a similar calculation done in Lemma \ref{lem:Xn},
\begin{align*}
\|\pull^{\sigma_n}_\uparrow(\nu_n)\vert_{MM} - \nu\vert_{MM}\|_{\mathrm{TV}}	&	\geq		\frac{1}{2}\sum_{w\in A^{MM} \setminus X_{MM}} \sum_{i=1}^{k_n} \alpha^n_i P^{\sigma_n}_{\fx_n^i}([w])	\\
															&	=	\frac{1}{2|V_n|} \sum_{i=1}^{k_n} \alpha^n_i \sum_{v \in V_n} \sum_{w\in A^{MM} \setminus X_{MM}}   \delta_{\pull_{v}^{\sigma_n}(\fx_n^i)}([w])	\\
															&	=	\frac{1}{2|V_n|} \sum_{i=1}^{k_n} \alpha^n_i |E(\fx_n^i)|,
\end{align*}
so
$$
\sum_{i=1}^{k_n} \alpha^n_i |E(x_n^i)| \leq 2|V_n|\|\pull^{\sigma_n}_\uparrow(\nu_n)\vert_{MM} - \nu\vert_{MM}\|_{\mathrm{TV}} \leq 2|V_n|\|\pull^{\sigma_n}_\uparrow(\nu_n)\vert_{F_{r_0}} - \nu\vert_{F_{r_0}}\|_{\mathrm{TV}}.
$$

Now, for every $\fx_n^i$, consider the corrected version $\fy_n^i \in X^n$ provided by Lemma \ref{lem:error}, and define $\tilde{\nu}_n \in \Prob(X^n)$ as $\sum_{i=1}^{k_n} \alpha^n_i \delta_{\fy^i_n}$. Suppose that $d(\pull^{\sigma_n}_\uparrow(\nu_n), \nu) \leq \delta'$, so $\pull^{\sigma_n}_\uparrow(\nu_n) \in O\left(M,\frac{\epsilon_0}{8|M|^2|F_{r_0}|},\nu\right)$. Then, again as in Lemma \ref{lem:Xn},
\begin{align*}
\|\pull^{\sigma_n}_\uparrow(\nu_n)\vert_{F_{r_0}} - \pull^{\sigma_n}_\uparrow(\tilde{\nu}_n)\vert_{F_{r_0}}\|_{\mathrm{TV}}	&	\leq	\sum_{i=1}^{k_n} \alpha^n_i \|P^{\sigma_n}_{\fx_n^i}\vert_{F_{r_0}} - P^{\sigma_n}_{\fy_n^i}\vert_{F_{r_0}}\|_{\mathrm{TV}}	\\
																			&	\leq	\sum_{i=1}^{k_n} \alpha^n_i \frac{1}{|V_n|} (|V_n \setminus V^{MMF_{r_0}}_n| + |M|^2|F_{r_0}||E(\fx_n^i)|)	\\
																			&	=	\frac{|V_n \setminus V^{MMF_{r_0}}_n|}{|V_n|} + \frac{|M|^2|F_{r_0}|}{|V_n|} \sum_{i=1}^{k_n} \alpha^n_i  |E(\fx_n^i)|	\\
																			&	\leq	\frac{\epsilon_0}{4} + 2|M|^2|F_{r_0}|\|\pull^{\sigma_n}_\uparrow(\nu_n)\vert_{F_{r_0}} - \nu\vert_{F_{r_0}}\|_{\mathrm{TV}}	\\
																			&	\leq	\frac{\epsilon_0}{4} + \frac{\epsilon_0}{4} = \frac{\epsilon_0}{2}.
\end{align*}

Therefore,
\begin{align*}
\|\pull^{\sigma_n}_\uparrow(\tilde{\nu}_n)\vert_{F_{r_0}} - \nu\vert_{F_{r_0}}\|_{\mathrm{TV}}	&	\leq	\|\pull^{\sigma_n}_\uparrow(\tilde{\nu}_n)\vert_{F_{r_0}} - \pull^{\sigma_n}_\uparrow(\nu_n)\vert_{F_{r_0}}\|_{\mathrm{TV}} + \|\pull^{\sigma_n}_\uparrow(\nu_n)\vert_{F_{r_0}} - \nu\vert_{F_{r_0}}\|_{\mathrm{TV}}	\\
&	\leq	\frac{\epsilon_0}{2} + \frac{\epsilon_0}{2} = \epsilon_0,
\end{align*}
i.e., $\pull^{\sigma_n}_\uparrow(\tilde{\nu}_n) \in O(r_0,\epsilon_0,\pull^{\sigma_n}_\uparrow(\nu_n)) \subseteq B_d(\nu,\delta)$, so $d(\pull^{\sigma_n}_\uparrow(\tilde{\nu}_n),\nu) < \delta$.

Now, without loss of generality, $\fx^i_n \neq \fx^j_n$ for all $i \neq j$. Then, we have that $H(\nu_n) = H(\alpha^n_1,\dots,\alpha^n_{k_n})$. We want to estimate $H(\tilde{\nu}_n)$. First, notice that it could be the case that $\fy^i_n = \fy^j_n$ for $i \neq j$, which would cause a drop in the entropy.

Since $\epsilon_ 0 < \delta$, we have that $\sum_{i=1}^{k_n} \alpha^n_i |E(\fx_n^i)| \leq \delta|V_n|$. Without loss of generality, we can suppose that there are $q_n$ points $\{\fy_n^{j}\}_{j=1}^{q_n}$ in $X^n$ and indices $1 = \ell_1 < \ell_2 < \dots < \ell_{q_n} < \ell_{q_n+1} = k_n + 1$ such that for every $1 \leq j \leq q_n$ and $\ell_j \leq i < \ell_{j+1}$, the correction of the point $\fx_n^i$ is the point $\fy_n^{j}$. If we define $C^n_j = \sum_{i=\ell_j}^{\ell_{j+1}-1} \alpha^n_i$ for $1 \leq j \leq q_n$, we have that $\tilde{\nu} = \sum_{j=1}^{q_n} C^n_{j} \fy_n^{j}$. Then, by the grouping property of Shannon entropy, we have that
$$
H(\alpha^n_1,\dots,\alpha^n_{k_n})	=	H(C^n_1, \dots, C^n_{q_n}) + \sum_{j=1}^{q_n} C^n_j \cdot H\left(\frac{\alpha^n_{\ell_j}}{C^n_j},\dots,\frac{\alpha^n_{\ell_{j+1}-1}}{C^n_j}\right).
$$

Now, define $L_\delta = \{1 \leq i \leq k_n: |E(x^i_n)| > \delta^{1/2} |V_n|\}$, which we regard as the set of indices of points with ``large errors''. Then,
$$
\delta |V_n| \geq \sum_{i \in L_\delta} \alpha^n_i |E(x^i_n)| \geq \delta^{1/2} |V_n| \sum_{i \in L_\delta} \alpha^n_i,
$$
so $\sum_{i \in L_\delta} \alpha^n_i \leq \delta^{1/2}$. On the other hand, for $1 \leq j \leq q_n$, there are at most ${|V_n| \choose \delta^{1/2} |V_n|} |A|^{\delta^{1/2} |V_n|}$ different points $\fx_n^i$ with $i \notin L_\delta$ whose correction is the point $\fy_n^j$. Therefore, for every $j$, after reordering the coefficients $\alpha_i^n$ if necessary, we can suppose that there exists $\ell_j  \leq t_j < \ell_{j+1}$ such that 
$$
t_j - \ell_j + 1 < {|V_n| \choose \delta^{1/2} |V_n|} |A|^{\delta^{1/2} |V_n|} \quad \text{ and } \quad \sum_{i = t_j+1}^{\ell_{j+1}-1} \alpha^n_i \leq \delta^{1/2}.
$$

Then, again by the grouping property of Shannon entropy and since the uniform distribution maximizes it,
\begin{align*}
\begin{array}{lll}
&		&	H\left(\frac{\alpha^n_{\ell_j}}{C^n_j},\dots,\frac{\alpha^n_{\ell_{j+1}-1}}{C^n_j}\right)	\\
&	\leq	&	H\left(\frac{\alpha^n_{\ell_j}}{C^n_j},\dots,\frac{\alpha^n_{t_j}}{C^n_j}, \sum_{i=t_j + 1}^{\ell_{j+1}-1}\frac{\alpha^n_i}{C^n_j}\right) + \left(\sum_{i=t_j + 1}^{\ell_{j+1}-1}\frac{\alpha^n_i}{C^n_j}\right) \cdot H\left(\frac{\alpha^n_{t_j+1}}{C^n_j},\dots,\frac{\alpha^n_{\ell_{j+1}-1}}{C^n_j}\right)	\\
&	\leq	&	\log(t_j - \ell_j + 1) + \frac{\delta^{1/2}}{C^n_j} \log |A|^{|V_n|}	\\
&	\leq	&	\log({|V_n| \choose \delta^{1/2} |V_n|} |A|^{\delta^{1/2} |V_n|}) + \frac{\delta^{1/2}}{C^n_j}{|V_n| \log |A|}\\
&	\leq	&	|V_n|H(\delta^{1/2}, 1-\delta^{1/2}) + \delta^{1/2}|V_n|\log |A|(1 + \frac{1}{C^n_j}) + o(|V_n|),
\end{array}
\end{align*}
where in the last inequality we have used Stirling's approximation as in Theorem \ref{thm:partition}. Therefore, we have that
\begin{align*}
\begin{array}{lll}
&		&	|H(\nu_n) - H(\tilde{\nu}_n)|	\\
&	=	&	H(\alpha^n_1,\dots,\alpha^n_{k_n}) - H(C^n_1, \dots, C^n_{q_n})	\\
&	=	&	\sum_{j=1}^{q_n} C^n_j \cdot H\left(\frac{\alpha^n_{\ell_j}}{C^n_j},\dots,\frac{\alpha^n_{\ell_{j+1}-1}}{C^n_j}\right)	\\
&	\leq	&	\sum_{j=1}^{q_n} C^n_j \left( |V_n|H(\delta^{1/2}, 1-\delta^{1/2}) + \delta^{1/2}|V_n|\log |A|(1 + \frac{1}{C^n_j}) + o(|V_n|)\right)	\\
&	\leq	&	|V_n|H(\delta^{1/2}, 1-\delta^{1/2}) + 2\delta^{1/2}|V_n|\log |A| + o(|V_n|),
\end{array}
\end{align*}
so
$$
|V_n|^{-1}|H(\nu_n) - H(\tilde{\nu}_n)| \leq	 H(\delta^{1/2}, 1-\delta^{1/2}) + 2\delta^{1/2}\log |A| + o(1) =: f(\delta).
$$

Since $f(\delta) \to 0$ as $\delta \to 0$, we conclude.
\end{proof}

It is well-known that if $\mu$ is the unique Gibbs measure for $\phi$, then $\mu$ is ergodic (see \cite[Lemma 1]{alpeev2015entropy} and \cite[Corollary 7.4]{georgii2011gibbs}). Considering this, we have the following theorem.

\begin{theorem}
\label{thm:entropy}
If $X$ satisfies the TSSM property and there is a unique Gibbs measure $\mu$ for $\phi$, then,
$$
 h_{\Sigma}(\Gamma \acts (X,\mu)) = \limsup_n \frac{H(\mu_n)}{|V_n|}.
$$
\end{theorem}

\begin{proof}
First, let's prove that $h_{\Sigma}(\Gamma \acts (X,\mu)) \geq \limsup_n \frac{H(\mu_n)}{|V_n|}$. By uniqueness, $\mu$ is ergodic,  and we have that $h_{\Sigma}(\Gamma \acts (X,\mu)) = \hat{h}_\Sigma(\mu)$. Then, it suffices to prove that for small $\delta > 0$,
$$
\hat{h}_{\Sigma,\delta}(\mu) \geq \limsup_n \frac{H(\mu_n)}{|V_n|},
$$
but this follows from the fact that, due to Proposition \ref{lwconv}, $\{\mu_n\}_n$ locally weak* converges to $\mu$, and for sufficiently large $n$, we have that $d(\pull^{\sigma_n}_\uparrow(\mu_n),\mu) < \delta$, so $\mu_n \in \Appr_{\Sigma, n, \delta}(\mu)$ and $\hat{h}_{\Sigma, n, \delta}(\mu) \geq \frac{H(\mu_n)}{|V_{n}|}$. Taking limit superior in $n$, we conclude.

Now, let's prove that for every $\epsilon>0$, there exists $\delta > 0$ and a positive integer $n_{0}$ such that for every $n \geq n_0$, we have
$$
\frac{H(\nu_n)}{|V_n|} \leq \limsup_n \frac{H(\mu_n)}{|V_n|} + \epsilon
$$
for any $\nu_n \in \Appr_{\Sigma, n, \delta}(\mu)$. In order to do this, choose $\delta > 0$ such that $f(\delta) \leq \frac{\epsilon}{4}$ and for any $\nu^1, \nu^2 \in \Prob(A^\Gamma)$ with $d(\nu^1, \nu^2) \leq 2 \delta$, the inequality
$$
\left|\int \phi(y) d\nu^1 - \int \phi(y) d\nu^2\right| \leq \frac{\epsilon}{4}
$$
holds. Notice that, since $\|\phi\|$ is finite, such $\delta$ exists. By Proposition \ref{prop:supp}, there exists $0 < \delta' < \delta$ and $n_0$ such that, for all $n \geq n_0$, if $\nu_n \in \Appr_{\Sigma, n, \delta'}(\mu)$, then we can find $\tilde{\nu}_n \in \Prob(X^n)$ such that $\tilde{\nu}_n \in \Appr_{\Sigma, n, \delta}(\mu)$ and $\frac{H(\nu_n)}{|V_n|} \leq \frac{H(\tilde{\nu}_n)}{|V_n|} + f(\delta) \leq \frac{H(\tilde{\nu}_n)}{|V_n|} + \frac{\epsilon}{5}$.

Consider $n \geq n_{0}$ such that $d(\pull^{\sigma_n}_\uparrow(\mu_n),\mu) < \delta$ and $\frac{H(\mu_n)}{|V_n|} \leq \limsup_k \frac{H(\mu_k)}{|V_k|} + \frac{\epsilon}{5}$. Pick $\nu_n \in \Appr_{\Sigma, n, \delta'}(\mu)$ such that $\hat{h}_{\Sigma, n,\delta'}(\mu) \leq \frac{H(\nu_n)}{|V_n|} + \frac{\epsilon}{5}$. By Lemma \ref{estimate}, for any $\tilde{\nu}_n \in \Prob(X^n)$, 
$$
\frac{H(\tilde{\nu}_n)}{|V_n|} + \int_X \phi(x) d(\pull^{\sigma_n}_\uparrow(\tilde{\nu}_n))(x)	\leq	\frac{H(\mu_n)}{|V_n|} + \int_X \phi(x) d(\pull^{\sigma_n}_\uparrow(\mu_n))(x)	
$$
and, since $d(\pull^{\sigma_n}_\uparrow(\tilde{\nu}_n),\pull^{\sigma_n}_\uparrow(\mu_n)) \leq 2 \delta$, it follows that
\begin{align*}
\frac{H(\tilde{\nu}_n)}{|V_n|}	\leq	\frac{H(\mu_n)}{|V_n|} + \int_{X} \phi(x) d(\pull^{\sigma_n}_\uparrow(\mu_n))(y)-\int_{X} \phi(x) d(\pull^{\sigma_n}_\uparrow(\tilde{\nu}_n))(x) 	\leq	\frac{H(\mu_n)}{|V_n|} + \frac{\epsilon}{5}.
\end{align*}

Then, combining all these inequalities, we obtain that,
\begin{align*}
\hat{h}_{\Sigma, n,\delta'}(\mu)	\leq	\frac{H(\nu_n)}{|V_n|} + \frac{\epsilon}{5} \leq	\frac{H(\tilde{\nu}_n)}{|V_n|} +  \frac{2\epsilon}{5} \leq	\frac{H(\mu_n)}{|V_n|} + f(\delta) +  \frac{3\epsilon}{5} \leq	\limsup_k \frac{H(\mu_k)}{|V_k|} + \epsilon.
\end{align*}

Taking limit superior in $n$, we get that
$$
\hat{h}_{\Sigma}(\mu) \leq \hat{h}_{\Sigma,\delta'}(\mu) = \limsup_n \hat{h}_{\Sigma,n,\delta'}(\mu) \leq \limsup_k \frac{H(\mu_k)}{|V_k|} + \epsilon.
$$
and since $\epsilon$ was arbitrary, we conclude.
\end{proof}


\begin{corollary}
\label{cor:eq}
If $X$ satisfies the TSSM property and $\mu$ is the unique Gibbs measure for $\phi$, then $\mu$ is an equilibrium state for $\phi$.
\end{corollary}

\begin{proof}
By definition of $Z_n$, for every $\fx \in X^n$,
$$
\log Z_n = -\log \mu_n(\fx) + \fEn_n(\fx).
$$

Integrating against $\mu_n$, we have that
\begin{align*}
\log Z_n	&	=	\int{(-\log \mu_n(\fx) + \fEn_n(\fx))}d\mu_n	\\	
		&	=	H(\mu_n) + \int{\fEn_n(\fx)}d\mu_n(\fx)	\\
		&	=	H(\mu_n) + \sum_{v \in V_n} \int{\phi(\pull^{\sigma_n}_v(\fx))}d\mu_n(\fx)	\\
		&	=	H(\mu_n) + |V_n|\int{\phi(x)}d(\pull^{\sigma_n}_\uparrow(\mu_n))(x),
\end{align*}
so
$$
\frac{H(\mu_n)}{|V_n|} = \log \frac{Z_n}{|V_n|} - \int{\phi(x)}d(\pull^{\sigma_n}_\uparrow(\mu_n))(x).
$$

By Theorem \ref{thm:partition} and Proposition \ref{lwconv},
\begin{align*}
\limsup_n \frac{H(\mu_n)}{|V_n|}	&	=	\limsup_n \log \frac{Z_n}{|V_n|} - \lim_n \int{\phi(x)}d(\pull^{\sigma_n}_\uparrow(\mu_n))(x)	\\
							&	=	p_{\Sigma}(\Gamma \acts X, \phi)  - \int{\phi(x)}d\mu(x),
\end{align*}
and, by Theorem \ref{thm:entropy} and ergodicity of $\mu$, we obtain that
\begin{align*}
h_{\Sigma}(\Gamma \acts (X,\mu))		=	\limsup_n \frac{H(\mu_n)}{|V_n|}	
								=	p_{\Sigma}(\Gamma \acts X, \phi)  - \int{\phi}d\mu,
\end{align*}
i.e.,
$$
p_\Sigma(X,\phi) = h_{\Sigma}(\Gamma \acts (X,\mu)) + \int{\phi}d\mu.
$$

Therefore, $\mu$ is an equilibrium state.

%
%
%
%
%
%
%
%

\end{proof}


\section{Strong spatial mixing}
\label{sec6}

Given $y,z \in X$ and $F \subseteq \Gamma$, define 
$$
\Delta_F(y,z) := \{g \in F: y(g) \neq z(g)\},
$$
i.e., the set of elements in $F$ where $y$ and $z$ differ. For $\beta: \N \to \R_{\geq 0}$ such that $\lim_r \beta(r) = 0$, we say that $\mu$ exhibits {\bf strong spatial mixing (SSM) with decay rate $\beta$} if for all $x, y ,z \in X$, $g \in \Gamma$, and $F \Subset \Gamma \setminus \{g\}$,
$$
\left|\mu([x_g] \vert [y_F]) - \mu([x_g] \vert [z_F])\right| \leq \beta(\dist(g,\Delta_F(y,z))),
$$
where $\dist(g,\Delta) := \min\{r \in \N: F_r g \cap \Delta \neq \emptyset\}$ for every $\Delta \subseteq \Gamma$. We say that $\mu$ exhibits {\bf strong spatial mixing (SSM)} if it exhibits SSM with decay rate $\beta$ for some $\beta$. This definition recovers the more usual definition of SSM in finitely generated groups where the distance is the graph distance in an associated Cayley graph and the exhaustive sequence $\{F_r\}_r$ is the sequence of $r$-balls centered at $1_\Gamma$ (e.g., see \cite{4-briceno, austin2018gibbs}).

If $\mu$ is fully supported, which is the case if $X$ satisfies the TSSM property or weaker mixing conditions, this is equivalent to the following condition at the level of the specification $\gamma$: for all $x, y ,z \in X$, $F \Subset \Gamma$, and $g \in F$, 
$$
\left|\gamma_F([x_g] \vert y) - \gamma_F([x_g] \vert z)\right| \leq \beta(\dist(g,\Delta_{F^c}(y,z))),
$$

In such case, if $\gamma$ admits a Gibbs measure $\mu$ that exhibits SSM, then $\mu$ is the unique Gibbs measure for $\phi$ (see \cite{3-weitz}).






\subsection{Some technical consequences}

Given $a,b \in \R$, we will sometimes write `$a \approx_{\varepsilon} b$' instead of `$|a-b|<\varepsilon$'.

\begin{lemma}
\label{lem:ssm1}
If $\mu$ exhibits SSM, then, for all $x, y, z \in X$, $g \in \Gamma$, and $D, E \Subset \Gamma \setminus \{g\}$ such that $y_{D \cap F_r g} = z_{E \cap F_r g}$, it follows that
$$
\left|\mu([x_g] \vert [y_D]) - \mu([x_g] \vert [z_E])\right| \leq 3\beta(r).
$$
\end{lemma}

\begin{proof}
Fix $x, y, z \in X$, $g \in \Gamma$, and $D, E \Subset \Gamma \setminus \{g\}$ such that $y_{D \cap F_r g} = z_{E \cap F_r g}$. By SSM and the Markovian property of $\mu$,
\begin{align*}
\mu([x_g] \vert [y_D])	&	=	\sum_w \mu([x_g] \vert [y_D w_{\partial_M(F_r g) \setminus D}])	\mu([w_{\partial_M(F_r g) \setminus D}] \vert [y_D])	\\
				&	=	\sum_w \mu([x_g] \vert [y_{D \cap MF_r g} w_{\partial_M(F_r g) \setminus D}])	\mu([w_{\partial_M(F_r g) \setminus D}] \vert [y_D])	\\
				&	\approx_{\beta(r)}	\sum_w \mu([x_g] \vert [y_{D \cap MF_r g}y_{\partial_M(F_r g) \setminus D}])	\mu([w_{\partial_M (F_r g) \setminus D}] \vert [y_D])	\\
				&	=				\mu([x_g] \vert [y_{D \cap F_r g}] \cap [y_{\partial_M F_r g}])	\\
				&	\approx_{\beta(r)}	\mu([x_g] \vert [z_{E \cap F_r g}] \cap [z_{\partial_M F_r}])	\\
				&	=	\sum_w \mu([x_g] \vert [z_{E \cap MF_r g}y_{\partial_M (F_r g) \setminus E}])	\mu([w_{\partial_M (F_r g) \setminus E}] \vert [z_F])	\\
				&	\approx_{\beta(r)}	\sum_w \mu([x_g] \vert [z_{E \cap MF_r g}w_{\partial_M (F_r g) \setminus E}])	\mu([w_{\partial_M(F_r g) \setminus E}] \vert [z_E])	\\
				&	=	\sum_w \mu([x_g] \vert [z_Ew_{\partial_M(F_r g) \setminus E}])	\mu([w_{\partial_M(F_r g) \setminus E}] \vert [z_E])	\\
				&	=	\mu([x_g] \vert [z_E]),
\end{align*}
where $\sum_w$ denotes the sum over all the partial configurations $w_{\partial_M (F_r g) \setminus D}$ and $w_{\partial_M (F_r g) \setminus E}$ such that $\mu\left([w_{\partial_M (F_r g) \setminus D}] \vert [y_D]\right) > 0$ and $\mu\left([w_{\partial_M (F_r g) \setminus E}] \vert [z_E]\right) > 0$, respectively.
\end{proof}

Given $r \in \N$ and $\epsilon > 0$, we define the set
$$
V_n^{r,\epsilon} := \{v \in V^r_{n}: \left|\mu_{n}\left([\fx_{v}] \vert [\fx_{U}]\right)-\mu_{n}\left([\fx_{v}] \vert [\fx_{U \cap \sigma_n^{F_r}(v)}]\right)\right| \leq \epsilon \text{ for } U \subseteq V_{n}, \fx \in X^n\}.
$$

%

\begin{lemma}
\label{lem:good}
If $\mu$ exhibits SSM, then, for every $\epsilon > 0$, there exists $r_0 \in \N$ such that, for every $r \geq r_0$, it holds that
$$
v \in V_n^{r,\epsilon}  \quad	\text{ w.h.p. in } v \in V_n.
$$
\end{lemma}

\begin{proof}
%

Fix $\epsilon > 0$. Choose $r_0 \in \N$ such that $MM \subseteq F_{r_0}$ and $\beta(r_0) \leq \epsilon/5$. We know that, for every $r \geq r_0$, $v \in V^r_n$, w.h.p. in $v \in V_n$. In addition, by SSM, we have uniqueness of Gibbs measures, and by local weak* convergence, w.h.p. in $v \in V_n$,
$$\left\|\left(\pull_{v}^{\sigma_n, r}\right)_{*} \mu_n\vert_{F_r} - \mu\vert_{F_r}\right\|_{\mathrm{TV}} \leq \epsilon/5.
$$

Then, for every $U \subseteq V_n$, w.h.p. in $v \in V_n$,
\begin{align*}
\begin{array}{ll}
		&	\mu_{n}\left([\fx_{v}] \vert [\fx_{U}]\right)		\\
	=	&	\sum_\fy \mu_{n}\left([\fx_{v}] \vert [\fx_{U}\fy_{\sigma_n^{MF_r \setminus F_r}(v) \setminus U}]\right)\mu_{n}\left([\fy_{\sigma_n^{MF_r \setminus F_r}(v) \setminus U}] \vert [\fx_{U}]\right)	\\
	=	&	\sum_\fy \mu_{n}\left([\fx_{v}] \vert [\fx_{U \cap \sigma_n^{MF_r}(v)}\fy_{\sigma_n^{MF_r \setminus F_r}(v) \setminus U}]\right)\mu_{n}\left([\fy_{\sigma_n^{MF_r \setminus F_r}(v) \setminus U}] \vert [\fx_{U}]\right)	\\
	=	&	\sum_\fy (\pull_v^{\sigma_n, MF_r})_*\mu_{n}\left([\pull_v^{\sigma_n, MF_r}(\fx)_{1_\Gamma}] \vert [\pull_v^{\sigma_n, MF_r}(\fx_{U \cap \sigma_n^{MF_r}(v)}\fy_{\sigma_n^{MF_r \setminus F_r}(v) \setminus U})]\right)	\\
		&	\qquad\mu_{n}\left([\fy_{\sigma_n^{MF_r \setminus F_r}(v) \setminus U}] \vert [\fx_{U}]\right)	\\
	\approx_\epsilon	&		\sum_\fy \mu\left([\pull_v^{\sigma_n, MF_r}(\fx)_{1_\Gamma}] \vert [\pull_v^{\sigma_n, MF_r}(\fx_{U \cap \sigma_n^{MF_r}(v)}\fy_{\sigma_n^{MF_r \setminus F_r}(v) \setminus U})]\right)	\\
		&	\qquad \mu_{n}\left([\fy_{\sigma_n^{MF_r \setminus F_r}(v) \setminus U}] \vert [\fx_{U}]\right)	\\
	\approx_{3\beta(r)}	&	\sum_\fy \mu\left([\pull_v^{\sigma_n, MF_r}(\fx)_{1_\Gamma}] \vert [\pull_v^{\sigma_n, MF_r}(\fx_{U \cap \sigma_n^{F_r}(v)})]\right)\mu_{n}\left([\fy_{\sigma_n^{MF_r \setminus F_r}(v) \setminus U}] \vert [\fx_{U}]\right)	\\
	=	&	\mu\left([\pull_v^{\sigma_n, MF_r}(\fx)_{1_\Gamma}] \vert [\pull_v^{\sigma_n, MF_r}(\fx_{U \cap \sigma_n^{MF_r}(v))})]\right)	\\
	\approx_\epsilon	&	\mu_n\left([\fx_{v}] \vert [\fx_{U \cap \sigma^{F_r}_n(v)}]\right),							
\end{array}
\end{align*}
where $\sum_\fy$ denotes the sum over all the partial configurations $\fy_{\sigma_n^{MF_r \setminus F_r}(v) \setminus U}$ such that
$$
\mu_{n}\left([\fy_{\sigma_n^{MF_r \setminus F_r}(v) \setminus U}] \vert [\fx_{U}]\right)> 0.
$$
\end{proof}

\begin{lemma}
\label{lem:cond}
If $\{\mu_n\}_n$ locally weak* converges to $\mu$, then, for every $r \in \N$ and $\epsilon > 0$,
$$
\lim_{n \rightarrow \infty} U_{n}\left\{v: \left| \left(\pull_{v}^{\sigma_{n},r}\right)_{*} \mu_{n}([x_E] \vert [x_F]) - \mu([x_E] \vert [x_F])\right| > \epsilon\right\} = 0
$$
uniformly in $x \in X$ and $E, F \subseteq F_r$.
\end{lemma}

\begin{proof}
Fix $r \in \N$ and $\epsilon > 0$. Without loss of generality, we can assume that $\epsilon < 2$. Consider the usual local weak* convergence with
$$
\epsilon' =  \left(\frac{\epsilon}{2+\epsilon}\right) \min_{x \in X, F \subseteq F_r} \{\left(\pull_{v}^{\sigma_{n},r}\right)_{*}\mu^n([x_F]), \mu([x_F])\}.
$$

Notice that $\epsilon' > 0$, since both $\mu_n$ and $\mu$ are fully supported on $X^n$ and $X$, respectively.

Then, w.h.p. in $v \in V_n$,
\begin{align*}
\begin{array}{lll}
&	&	\frac{\left(\pull_{v}^{\sigma_{n},r}\right)_{*}([x_E] \cap [x_F])}{\left(\pull_{v}^{\sigma_{n},r}\right)_{*}([x_F])} - \frac{\mu([x_E] \cap [x_F])}{\mu([x_F])}	\\
&=	&	\frac{\left(\pull_{v}^{\sigma_{n},r}\right)_{*}([x_E] \cap [x_F])\mu([x_F]) - \left(\pull_{v}^{\sigma_{n},r}\right)_{*}([x_F])\mu([x_E] \cap [x_F])}{\left(\pull_{v}^{\sigma_{n},r}\right)_{*}([x_F])\mu([x_F])}	\\
&\leq	&	\frac{(\mu([x_E] \cap [x_F])+\epsilon')\mu([x_F]) - (\mu([x_F])-\epsilon')\mu([x_E] \cap [x_F])}{(\mu([x_F]) - \epsilon')\mu([x_F])}	\\
&=	&	\frac{\epsilon'(\mu([x_F]) + \mu([x_E] \cap [x_F]))}{(\mu([x_F]) - \epsilon')\mu([x_F])}	\\
&\leq	&	\frac{2\epsilon'\mu([x_F])}{(\mu([x_F]) - \epsilon')\mu([x_F])}	\\
&\leq	&	\frac{2\epsilon'}{\mu([x_F]) - \epsilon'}	\\
&\leq	&	\epsilon,
\end{array}
\end{align*}
and the other inequality is analogous.
\end{proof}

\subsection{Uniform bounds}

Given a Gibbs measure $\mu$ and the sequence of derived Gibbs measures $\{\mu_n\}_n$ for $\phi$, we define
$$
c(\mu) := \inf_{x \in X} \inf_{g \in \Gamma} \inf_{F \Subset \Gamma} \mu([x_{g}] \vert [x_F]), \quad c(\{\mu_n\}_n) := \inf_n \inf_{\fx \in X^n} \inf_{v \in V_n} \inf_{U \subseteq V_n} \mu_n([\fx_v] \vert [\fx_U]),
$$
and
$$
c(\phi):= \min\{c(\mu),c(\{\mu_n\}_n)\}.
$$

\begin{lemma}
\label{lem:unifbound}
If $X$ satisfies the TSSM property, then $c(\phi)> 0$.
\end{lemma}

\begin{proof}

Fix $n \in \N$. Notice that the consistency at $v \in V_n$ and the value of $\phi(\Pi^{\sigma_n}_v(\fx))$ depend only on the set $\sigma^{MM}_n(v)$. In particular, conditioned on a cylinder set supported on $W_v := \sigma^{MM}_n(\sigma^{MM}_n(v))$, the value at $v$ is independent of the value at the elements in $V_n \setminus W_v$.

Consider arbitrary $v \in V_n$, $\fx \in X^n$, and $U \subseteq V_n$. Then, by a counting argument, it must exist $w \in A^{\sigma^{MM}_n(W_v) \setminus W_v}$ such that $\mu_n([w] \vert [\fx_U]) \geq |A|^{-|M|^6}$. Moreover, $\fx_v w \fx_U$ is a partial configuration that is locally consistent, since $\mu_n([w\fx_U]) > 0$ and, for every $u \in V_n^{MM}$, $\sigma^{MM}_n(u) \cap W_v \cap \{v\} = \emptyset$, so $[\Pi_{u}^{\sigma_n,MM}(w\fx_{U \cup \{v\}})] \neq \emptyset$. In particular, by Lemma \ref{lem:extend}, $\fx_v w \fx_U$ can be extended to $X^n$. 

Without loss of generality, $v \notin U$. Then, by considering the Markovian property of $\mu_n$ described above, we have that
\begin{align*}
\mu_n([\fx_v] \vert [\fx_U])	&	\geq	\mu_n([\fx_v] \vert [\fx_Uw])	\mu_n([w] \vert [\fx_U])	\\
					&	\geq	\mu_n([\fx_v] \vert [(\fx_Uw)_{\sigma^{MM}_n(W_v) \setminus W_v}]) |A|^{-|M|^6}	\\
					&	\geq	\frac{1}{|A|^{|M|^4}} \exp(-2\|\phi\||M|^4)|A|^{-|M|^6} > 0.
\end{align*}

The argument for $\mu$ is analogous and the same lower bound works for it.

\end{proof}

We also have, provided SSM is satisfied, a converse of the previous lemma.

\begin{lemma}
\label{lem:tssm}
Suppose that $X$ is a subshift, $\mu$ is fully supported and satisfies SSM, and $c(\mu) > 0$. Then, $X$ satisfies the TSSM property.
\end{lemma}

\begin{proof}
Take $r_{0} \in \N$ such that $\beta(r) < c(\mu)$ for all $r \geq r_{0}$. Suppose that $X$ does not satisfy the TSSM property. Then, there exists $g, h \in \Gamma$ with $\dist(g, \{h\}) = r \geq r_{0}$, $F \Subset \Gamma$, and points $x,y,z \in X$ such that $[x_g y_F], [y_F z_h] \neq \emptyset$, but $[x_g y_F z_h] =\emptyset$. Since $[x_g y_F] \neq \emptyset$, there exists $\tilde{z} \in X$ such that $[x_g y_F \tilde{z}_h] \neq \emptyset$. In particular, $\mu([x_g] \vert [y_F \tilde{z}_h]) \geq c(\mu)$, so, by SSM,
$$
c(\mu) \leq	 \mu([x_g] \vert [y_F \tilde{z}_h]) = |\mu([x_g] \vert [y_F \tilde{z}_h]) - \mu([x_g] \vert [y_F z_h])| \leq \beta\left(r_{0}\right) < c(\mu),
$$
which is a contradiction. Therefore, $X$ satisfies the TSSM property with gap at most $r_0$.
\end{proof}


\section{Ordered sofic approximations}
\label{sec7}




If $(\Omega, \mathcal{F}, \Pb)$ is a probability space, then a random variable on $\Omega$ is a measurable function $X: (\Omega, \mathcal{F}) \to S$ to a measurable space $S$ and the {\bf law} of $X$ is the probability measure $\Pb X^{-1}: S \to \R$ defined by $\Pb X^{-1}(s) = \Pb\left(X^{-1}(s)\right)$.

Given two random variables $X_1$ and $X_2$ defined on probability spaces $(\Omega_1, \mathcal{F}_1, \Pb_1)$ and $(\Omega_2, \mathcal{F}_2, \Pb_2)$, respectively, a {\bf coupling} of $X_{1}$ and $X_{2}$ is a new probability space $(\Omega, \mathcal{F}, \Pb)$ over which there are two random variables $Y_{1}$ and $Y_{2}$ such that $Y_{1}$ has the same distribution as $X_{1}$ while $Y_{2}$ has the same distribution as $X_{2}$.

The total variation distance between two random variables $X$ and $Y$ is defined as
$$
d_{\mathrm{T} V}(X, Y)=\sup _{A}\{|\Pb(X \in A)-\Pb(Y \in A)|\}
$$
where the supremum is over all (measurable) sets $A$. A classical theorem relating the total variation distance to couplings establishes that, for any two random variables $X$ and $Y$, there exists a coupling such that $\Pb(X \neq Y)=$ $d_{\mathrm{TV}}(X, Y)$. This coupling is known as the {\bf optimal coupling}.

\subsection{Random pasts}

Given a countable set $V$, consider a random variable $\past: \Omega \to 2^{V \times V}$ and its projections $\past^v: \Omega \to 2^{\{v\} \times V}$ for every $v \in V$, where we also understand $\past$ as a random function $\past: V \to 2^{V}$ and $\past^v$ as a random subset $\past^v \subseteq V$. We say that such a random variable $\past$ is a {\bf random past} on $V$ if, for almost every instance of $\past$,
\begin{enumerate}
\item for all $v \in V$ the condition $v \notin \past^v$ holds;
\item for all $u, v \in V$, if $u \in \past^{v}$, then $\past^u \subseteq \past^v$; and,
\item if $u \neq v$, then either $u \in \past^v$ or $v \in \past^u$,
\end{enumerate}
where we interpret $\past^v$ as the \emph{past of $v$}. Notice that a random past can be turned into a \emph{random order} $\prec$ of $V$ and vice versa by declaring $u \prec v$ if and only if $u \in \past^v$.


If $\Gamma$ is a countable group, a random past $\past: \Gamma \to 2^\Gamma$ is {\bf invariant} if $\past^g$ has the same distribution as $\past^{1_\Gamma}g$. A deterministic invariant random past on a group $\Gamma$ is called {\bf algebraic past}. A group with an algebraic past is called {\bf orderable}.



\subsection{Ordered sofic approximations}

Let $\past$ be an invariant random past on $\Gamma$. Given $v \in V_n$ and $U \subseteq V_n$ we define $\sigma_v^{-1}(U) := \{g \in \Gamma: \sigma^g(v) \in U\} \subseteq \Gamma$ for $U \subseteq V_n$. Notice that $\sigma_v^{-1}\past_n^v: \Omega_n \to 2^\Gamma$ is measurable. We say that a sofic approximation $\Sigma = \{\sigma_n: \Gamma \to \Sym(V_n)\}_n$ of $\Gamma$ is {\bf ordered (relative to $\past$)} if there exists a sequence of random pasts $\past_n: \Omega_n \to 2^{V_n \times V_n}$ such that for every $F \Subset \Gamma$ and $\delta>0$, there exists $n_0 \in \N$ such that for every $n \geq n_0$ and for every $v \in V_n$, there exist couplings between $\past^{1_\Gamma}$ and $\sigma_v^{-1}(\past_n^v)$ such that
$$
|\{v \in V_n: \Pb_n^{v,F,\delta}(\past^{1_\Gamma} \cap F = \sigma_v^{-1}(\past_n^v) \cap F) \leq 1-\delta\}| \leq \delta |V_n|,
$$
where $\Pb_n^{v,F,\delta}$ denotes the probability measure associated to the coupling between $\past^{1_\Gamma}$ and $\sigma_v^{-1}(\past_n^v)$ for the parameters $F$ and $\delta$.
\begin{remark}
Notice that establishing that a sofic approximation $\Sigma$ is ordered relative to an invariant random past $\past$ only depends on the distribution of the random set $\past^{1_\Gamma}$, and, a priori, one could attempt to order a sofic approximation $\Sigma$ relative to any random set $\past_0$. However, due to the very same definition, it can be proven that if a sofic approximation $\Sigma$ is ordered relative to $\past_0$, then the random function $\tilde{\past}: V \to 2^{V}$ defined as $\tilde{\past}(g) = \past_0 g$ turns out to be an invariant random past, so essentially the only way to order a sofic approximation $\Sigma$ is relative to an invariant random past. We leave this fact as an exercise.
\end{remark}

\subsection{Examples}

Here we present two classes of invariant random pasts $\past$ for which there exist sofic approximations that can be ordered relative to them.

\subsubsection{Invariant random pasts on amenable groups}
\label{exmp:amen}

Given any invariant random past $\past$ on an amenable group $\Gamma$ and a F{\o}lner sequence $\{T_n\}_n$ for $\Gamma$, define $\past_n := \pi_n \circ \past$, where $\pi_n: 2^{\Gamma \times \Gamma} \to 2^{T_n \times T_n}$ denotes the natural projection. Notice that, if  $F \Subset \Gamma$ and $Fh \subseteq T_n$, then
\begin{align*}
\sigma_h^{-1}(\past_n^h) \cap F	&	=	\{g \in F: \sigma^g(h) \in (\pi_n \circ \past)^h\}	\\
							&	=	\{g \in F: gh \in T_n \cap \past^h\}			\\
							&	=	F \cap \past^h h^{-1},
\end{align*}
and since $\past^h$ has the same distribution as $\past^{1_\Gamma}h^{-1}$, it suffices to consider the optimal coupling $\Pb^h$ between $\past^h$ and $\past^{1_\Gamma}h^{-1}$ to obtain that
$$
\Pb^h(\past^{1_G} \cap F = \sigma_h^{-1}(\past_n^h) \cap F) = 1,
$$
for every $Fh \subseteq T_n$. Then, by amenability, $\lim_n |T_n|^{-1}|\{h \in T_n: Fh \subseteq T_n\}| = 1$, so for every $F \Subset \Gamma$ and for every $\delta > 0$, there exists $n_0 \in \N$ such that for every $n \geq n_0$, there exists a family of couplings $\{\Pb^h\}_{h \in \Gamma}$ with
$$
|\{h \in T_n: \Pb^h(\past^{1_G} \cap F = \sigma_h^{-1}(\past_n)^h \cap F) = 1\}| \leq \delta |T_n|,
$$
which is enough to conclude that the sofic approximation associated to $\{T_n\}_n$ can be ordered relative to $\past$.

%

%
%
%
%
%
%

\subsubsection{Percolation past coupling}

Given a countable set $V$, we define the \emph{percolation past} $\past_{\mathrm{perc}}: V \to 2^{V}$ to be the random past induced by an i.i.d. collection $\{\chi_{v}\}_{v \in V}$ of random variables with uniform distribution on $[0,1]$ and $\past^v_{\mathrm{perc}} = \{u \in V: \chi_{u} < \chi_{v}\}$. Notice that if we consider this random past in a countable group $\Gamma$ it turns out to be invariant.


Let $\Sigma$  be an arbitrary sofic approximation to $\Gamma$. Then, given $F \Subset \Gamma$ and $\delta>0$, consider $n_0 \in \N$ such that for every $n \geq n_0$, $|V_n \setminus V_n^F| \leq \delta V_n$. Next, if $v \in V^F_n$, define $\Pb_n^{v,F}$ to be the \emph{diagonal coupling} $\Pb_n^{v,F}$ between $\{\chi_g\}_{g \in F}$ and $\{\chi_{\sigma_n^g(v)}\}_{g \in F}$ (i.e., $\chi_g$ and $\chi_{\sigma_n^g(v)}$ are identical almost surely) and, if $v \in V_n \setminus V^F_n$, define $\Pb_n^{v,F}$ to be the \emph{independent coupling} (i.e., $\chi_g$ and $\chi_{\sigma_n^g(v)}$ are independent). Then,
$$
|\{v \in V_n: \Pb_n^{v,F}(\past_{\mathrm{perc}}^{1_\Gamma} \cap F = \sigma_v^{-1}(\past_n^v) \cap F) < 1 \}| \leq \delta |V_n|,
$$
and we conclude. Notice that this shows that any sofic approximation can be ordered relative to the percolation past.


\begin{question}
Given an invariant random past $\past$ on a sofic group, when does it exist a sofic approximation $\Sigma$ that can be ordered relative to $\past$?
\end{question}

%
%

\section{Kieffer-Pinsker type formulas for pressure}
\label{sec8}

\subsection{Information and Shannon entropy}

Let's consider a sub-$\sigma$-algebra $\mathcal{F}$ of the Borel $\sigma$-algebra $\mathcal{B}$ in $X$ and a partition $\alpha$ of $X$. Given $\nu \in \Prob(X)$, we define the {\bf conditional information function} $I(\alpha | \mathcal{F}): X \to \R$ as
$$
I_\nu(\alpha | \mathcal{F})(x)=-\log \nu\left(A_{x} | \mathcal{F}\right)(x),
$$
where $A_{x}$ is the atom of $\alpha$ containing $x$. The {\bf conditional Shannon entropy of $\alpha$ with respect to $\mathcal{F}$} is defined as
$$
H_\nu(\alpha \vert \mathcal{F})=\int{I_\nu(\alpha | \mathcal{F})}d\nu.
$$

If $\beta$ is a partition, then we identify $\beta$ with the $\sigma$-algebra equal to the set of all unions of partition elements of $\beta$. Through this identification, $I(\alpha | \beta)$ and $H(\alpha | \beta)$ are well-defined. We denote by $H_\nu(\alpha)$ the conditional entropy of $\alpha$ with respect to the trivial $\sigma$-algebra $\{X,\emptyset\}$, i.e., $H_\nu(\alpha) := H_\nu(\alpha \vert \{X,\emptyset\})$. This corresponds to the usual Shannon entropy of $(\nu(A))_{A \in \alpha}$, i.e., $H_\nu(\alpha) = -\sum_{A \in \alpha}\nu(A)\log \nu(A)$. In \cite{1-alpeev}, it was proven the following theorem.

\begin{theorem}
Let $\Gamma \acts (X, \nu)$ be a probability measure preserving action of a countable amenable group $\Gamma$ and let $\past: \Gamma \rightarrow 2^{\Gamma}$ be an invariant random past on $\Gamma$. Suppose that $\alpha$ is a Borel partition with $H_{\mu}(\alpha) < +\infty$ and that $\mathcal{A}$ is a $\Gamma$-invariant $\sigma$-algebra on $X$. Then the following holds:
$$
h({\Gamma} \acts (X, \mu), \alpha | \mathcal{A}) = \mathbb{E}_{\past} H_{\mu}(\alpha | \alpha^{\past^{1_\Gamma}} \vee \mathcal{A}),
$$
where $h({\Gamma} \acts (X, \mu), \alpha | \mathcal{A})$ denotes the Kolmogorov-Sinai entropy of the partition $\alpha$ relative to $\mathcal{A}$.
\end{theorem}

In particular, if $\mathcal{A}$ is trivial, the previous theorem tells us that $h({\Gamma} \acts (X, \mu), \alpha) = \mathbb{E}_{\past} H_{\mu}(\alpha | \alpha^{\past^{1_\Gamma}})$. We aim to take this formula as a model for developing new formulas for entropy and pressure in the sofic and symbolic setting.

\subsection{Random information functions}

Suppose that $X$ satisfies the TSSM property and $\mu$ is the unique Gibbs measure for $\phi$ that satisfies SSM. Consider an exhaustion $\{F_r\}_r$ and suppose that $\alpha$ is the canonical partition of $X$, i.e., $\alpha = \{[x_{1_\Gamma}]: x \in X\}$. Given $r \in \N$, define the function $f_r: X \times 2^\Gamma \to \R$ given by $f_r(x,D) = I_\mu\left(\alpha | \alpha^{D \cap F_r}\right)(x)$, i.e.,
$$
f_r(x,D) = -\log \mu\left([x_{1_\Gamma}] \vert [x_{D \cap F_r}]\right) = \sum_{S \subseteq F_r} -\log\mu([x_{1_\Gamma}] \vert x_{S}]) 1_{\{D: S = D \cap F_r\}}.
$$

Notice that each function $f_r$ is continuous: $1_{\{D: S = D \cap F_r\}}$ is the indicator function of a clopen set; by SSM, $\mu([x_{1_\Gamma}] \vert x_{S}])$ and $\mu([y_{1_\Gamma}] \vert y_{S}])$ will be close in value if $x_{F_r} = y_{F_r}$ for large enough $r$; by the TSSM property, $\mu([x_{1_\Gamma}] \vert x_{S}]) \geq c(\phi)> 0$; and $\log: [c,+\infty) \to \R$ is Lipschitz (with Lipschitz constant $1/c$) for any $c > 0$.

In addition, for a fixed $(x,D) \in X \times 2^\Gamma$, by SSM, the sequence $\{f_r(x,D)\}_r$ is Cauchy. Therefore, the pointwise limit $f(x,D) := \lim_r f_r(x,D)$ exists everywhere. Moreover, $f$ is the uniform limit of $\{f_r\}_r$, since, for every $r_0 \in \N$ and $r,m \geq r_0$, by Lemma \ref{lem:ssm1} and the Lipschitz property of $\log(\cdot)$,
\begin{align*}
|f_r(x,D) - f_m(x,D)|	&	=	|\log \mu\left([x_{1_\Gamma}] \vert [x_{D \cap F_r}]\right) - \log \mu\left([x_{1_\Gamma}] \vert [x_{D \cap F_m}]\right)|		\\
				&	\leq	\frac{1}{c(\phi)}|\mu\left([x_{1_\Gamma}] \vert [x_{D \cap F_r}]\right) - \mu\left([x_{1_\Gamma}] \vert [x_{D \cap F_m}]\right)|	\\
				&	\leq	\frac{3}{c(\phi)}\beta(r_0),
\end{align*}
which can be made arbitrarily small uniformly in $x$ and $D$. Therefore, $f(x,D)$ is defined everywhere, non-negative, and continuous.

Given a random set $\mathcal{S} \subseteq 2^\Gamma$, we define the {\bf $\mathcal{S}$-random $\mu$-information function} as $I^{\mathcal{S}}_\mu: X \to \R$, with
$$
I_\mu^{\mathcal{S}}(x) := \int{I_{\mu}(x,D})d\lambda_{\mathcal{S}}(D),
$$
where $\lambda_{\mathcal{S}}$ denotes the law of $\mathcal{S}$. Notice that $I_\mu^{\mathcal{S}}$ is also continuous. Indeed, due to SSM,
$$
\left|I_\mu^{\mathcal{S}}(x) - I_\mu^{\mathcal{S}}(y)\right| \leq \int{|I_{\mu}(x,D) - I_{\mu}(y,D)|}d\lambda_{\mathcal{S}}(D) \leq \int{\frac{\beta(r)}{c(\phi)}}d\lambda_{\mathcal{S}}(D) = \frac{\beta(r)}{c(\phi)},
$$
provided $x$ and $y$ agree on $F_r$. We summarize these observations in the following proposition.

\begin{proposition}
Let $\mu \in \Prob(X)$ be a Gibbs measure for a finite range potential $\phi$. Suppose that $\phi$ satisfies SSM and $X$ satisfies the TSSM property. Then, for every random set $\mathcal{S}$, the $\mathcal{S}$-random $\mu$-information function $I^\mathcal{S}_\mu$ is defined everywhere and continuous. 
\end{proposition}

We will be mainly interested in $\mathcal{S}$-random $\mu$-information functions induced by random pasts. Given a random past $\past$, we will denote by $\lambda_{\past}$ the law of $\past$, by $\lambda^v_{\past}$ the law of $\past^v$, and write
$$
\mathbb{E}_{\past}[X] = \int{X}d\lambda_{\past},
$$
for any random variable $X$. Then, we define $I_\mu^\past$, the $\past$-random $\mu$-information function, as
$$
I_\mu^\past(x) = \mathbb{E}_{\past}[I_{\mu}(x,\past^{1_\Gamma})] = \mathbb{E}_{\past^{1_\Gamma}}[I_{\mu}(x,\past^{1_\Gamma})].
$$

In addition, if $F = F_r$, we denote by $\Pb_n^{v,r,\delta}$ the coupling $\Pb_n^{v,F_r,\delta}$ between $\past^{1_\Gamma}$ and $\sigma_v^{-1}(\past_n^v)$, and write
$$
\mathbb{E}_{n,v,r,\delta}[X] = \int{X}d\Pb_n^{v,r,\delta},
$$
for any random variable $X$.

\subsection{Main theorem}

%

Considering all the previous discussion, we have the main theorem of this work.

\begin{theorem}
\label{thm:main}
Suppose that $X$ satisfies the TSSM property and there is a unique Gibbs measure $\mu$ for $\phi$ that satisfies SSM. Then, for every invariant random past $\past$ and for every sofic approximation $\Sigma$ that can be ordered relative to $\past$,
$$
p_{\Sigma}(\Gamma \acts X,\phi) = \int{(I^\past_{\mu} + \phi)}d\nu,
$$
for all $\nu \in \Prob(X,\Gamma)$ such that $h_\Sigma(\Gamma \acts (X,\nu)) \neq -\infty$.
\end{theorem}

\begin{proof}
Since $h_\Sigma(\Gamma \acts (X,\nu)) \neq -\infty$, there exists a sequence $\fx_n \in A^{V_n}$ such that $P_{\fx_n}^{\sigma_n}$ weak* converges to $\nu$. Moreover, by Lemma \ref{lem:Xn}, we can assume that $\fx_n \in X^n$ for all $n$.


Recall that $\mu_{n}(\fx) = 1_{X^n}(\fx) Z_{n}^{-1} \exp \left\{\fEn_{n}(\fx)\right\}$ and the support of $\mu_n$ is $X^n$. In particular, for every $\fx_n \in X^n$,
$$
\log Z_{n} = -\log\mu_{n}(\fx_n) + \sum_{v \in V_n} \phi\left(\pull_{v}^{\sigma_n}(\fx_n)\right).
$$

We proceed to study the term $-\log\mu_{n}(\fx_n)$. Notice that for almost all instances of $\past_n$, it is induced a total order $\prec$ of $V_n$ with $u \prec v$ if and only if $\past_n^u \subseteq \past_n^v$. Then, for any such an instance, we have that
$$
\mu_n(\fx) = \prod_{v \in V_n} \mu_{n}(\fx_v \vert \fx_{\past_n^v}).
$$

Then,
\begin{align*}
-\log\mu_{n}(\fx)	&	=	\mathbb{E}_{\past_n} [-\log \mu_n(\fx)]	\\
				&	=	\mathbb{E}_{\past_n}[\sum_{v \in V_n} -\log\mu_{n}([\fx_v] \vert [\fx_{\past_n^v}])]	\\
				&	=	-\sum_{v \in V_n} \mathbb{E}_{\past_n}[\log\mu_{n}([\fx_v] \vert [\fx_{\past_n^v}])]	\\
				&	=	-\sum_{v \in V_n} \mathbb{E}_{\past_n^v}[\log\mu_{n}([\fx_v] \vert [\fx_{\past_n^v}])].
\end{align*}

Fix $\epsilon > 0$. By Lemma \ref{lem:good}, we can choose $r \in \N$ so that $\beta(r) \leq \epsilon$ and $v \in V^{r,\epsilon}_n$, w.h.p. in $v \in V_n$. Then, for every $v \in V^{r,\epsilon}_n$,
$$
\mu_{n}([\fx_v] \vert [\fx_{\past_n^v}]) \approx_\epsilon \mu_{n}([\fx_v] \vert [\fx_{\past_n^v \cap \sigma_n^{F_r}(v)}]),
$$
and, by Lemma \ref{lem:unifbound} and the Lipschitz property of $\log(\cdot)$, for every $v \in V^{r,\epsilon}_n$,
$$
\log\mu_{n}([\fx_v] \vert [\fx_{\past_n^v}]) \approx_{\frac{\epsilon}{c(\phi)}} \log\mu_{n}([\fx_v] \vert [\fx_{\past_n^v \cap \sigma_n^{F_r}(v)}]).
$$

On the other hand, again by Lemma \ref{lem:unifbound}, for every $v \in V_n \setminus V^{r,\epsilon}_n$,
$$
0 \leq -\log\mu_{n}([\fx_v] \vert [\fx_{\past_n^v}]) \leq -\log c(\phi),
$$
and, if we assume $n$ is large enough so that $|V_n \setminus V^{r,\epsilon}_n| \leq \epsilon |V_n|$, then
\begin{align*}
\frac{1}{|V_n|} \log\mu_{n}(\fx)	&	\approx_{\epsilon(\frac{1}{c(\phi)} - 2\log c(\phi))} \frac{1}{|V_n|} \sum_{v \in V_n} \mathbb{E}_{\past_n^v}[\log\mu_{n}([\fx_v] \vert [\fx_{\past_n^v \cap \sigma_n^{F_r}(v)}])].
\end{align*}

Notice that, for all $v \in V_n^r$, we have that $\mu_{n}([\fx_{\sigma_n^{F_r}(v)}]) = (\Pi_v^{\sigma_n, r})_*\mu_{n}([\Pi_v^{\sigma_n, r}(\fx)_{F_r}])$. In particular, if $\past^{1_\Gamma} \cap F_r = \sigma_n^{-1}(\past_n^v) \cap F_r$, it follows that
$$
\mu_{n}([\fx_v] \vert [\fx_{\past_n^v \cap \sigma_n^{F_r}(v)}]) = (\Pi_v^{\sigma_n, r})_*\mu_{n}([\Pi_v^{\sigma_n, r}(\fx)_{1_\Gamma}] \vert [\Pi_v^{\sigma_n, r}(\fx)_{\past^{1_\Gamma} \cap F_r}]),
$$
so, w.h.p. in $v \in V_n$,
$$
\Pb_n^{v,r,\epsilon}(\mu_{n}([\fx_v] \vert [\fx_{\past_n^v \cap \sigma_n^{F_r}(v)}]) \neq (\Pi_v^{\sigma_n, r})_*\mu_{n}([\Pi_v^{\sigma_n, r}(\fx)_{1_\Gamma}] \vert [\Pi_v^{\sigma_n, r}(\fx)_{\past^{1_\Gamma} \cap F_r}])) \leq \epsilon.
$$

Now, since $\mu_n$ locally weak* converges to $\mu$ and thanks to Lemma \ref{lem:cond}, we have that, w.h.p. in $v \in V_n$,
$$
(\Pi_v^{\sigma_n, r})_*\mu_{n}([\Pi_v^{\sigma_n, r}(\fx)_{1_\Gamma}] \vert [\Pi_v^{\sigma_n, r}(\fx)_{\past^{1_\Gamma} \cap F_r}]) \approx_\epsilon \mu([\Pi_v^{\sigma_n, r}(\fx)_{1_\Gamma}] \vert [\Pi_v^{\sigma_n, r}(\fx)_{\past^{1_\Gamma} \cap F_r}]).
$$

On the other hand, by SSM and Lemma \ref{lem:ssm1},
$$
\mu([\Pi_v^{\sigma_n, r}(\fx)_{1_\Gamma}] \vert [\Pi_v^{\sigma_n, r}(\fx)_{\past^{1_\Gamma} \cap F_r}]) \approx_{3\beta(r)} \mu([\Pi_v^{\sigma_n, r}(\fx)_{1_\Gamma}] \vert [\Pi_v^{\sigma_n, r}(\fx)_{\past^{1_\Gamma}}]).
$$
and by Lemma \ref{lem:unifbound} and the Lipschitz property of $\log(\cdot)$,
$$
\log\mu([\Pi_v^{\sigma_n, r}(\fx)_{1_\Gamma}] \vert [\Pi_v^{\sigma_n, r}(\fx)_{\past^{1_\Gamma} \cap F_r}]) \approx_{\frac{3\beta(r)}{c(\phi)}} \log\mu([\Pi_v^{\sigma_n, r}(\fx)_{1_\Gamma}] \vert [\Pi_v^{\sigma_n, r}(\fx)_{\past^{1_\Gamma}}]).
$$

Combining all these partial results, we have that, w.h.p. in $v \in V_n$,
\begin{align*}
\begin{array}{ll}
&\mathbb{E}_{\past_n^v}[\log\mu_{n}([\fx_v] \vert [\fx_{\past_n^v \cap \sigma_n^{F_r}(v)}])]		\\
	=							& \mathbb{E}_{n,v,r,\delta}[\log\mu_{n}([\fx_v] \vert [\fx_{\past_n^v \cap \sigma_n^{F_r}(v)}])]	\\
	\approx_{-2\epsilon\log c(\phi)}		&	\mathbb{E}_{n,v,r,\delta}[\log (\Pi_v^{\sigma_n, r})_*\mu_{n}([\Pi_v^{\sigma_n, r}(\fx)_{1_\Gamma}] \vert [\Pi_v^{\sigma_n, r}(\fx)_{\past^{1_\Gamma} \cap F_r}])]	\\
	\approx_{\frac{\epsilon}{c(\phi)}}		&	\mathbb{E}_{\past}[\log\mu([\Pi_v^{\sigma_n, r}(\fx)_{1_\Gamma}] \vert [\Pi_v^{\sigma_n, r}(\fx)_{\past^{1_\Gamma} \cap F_r}])]	\\
	\approx_{\frac{3\epsilon}{c(\phi)}}	&	\mathbb{E}_{\past}[\log\mu([\Pi_v^{\sigma_n}(\fx)_{1_\Gamma}] \vert [\Pi_v^{\sigma_n}(\fx)_{\past^{1_\Gamma}}])]		\\
	=							&	-I_\mu^{\past}(\Pi_v^{\sigma_n}(\fx)).
\end{array}
\end{align*}

Therefore, 
\begin{align*}
\log Z_n	&	\approx_{f(\epsilon)} \sum_{v \in V_n} I_\mu^{\past}(\pull_v^{\sigma_n}(\fx_n)) + \sum_{v \in V_n} \phi\left(\pull_{v}^{\sigma_n}(\fx_n)\right) + o(|V_n|)	\\
 		&	= 		|V_n|\int{I_\mu^{\past}(x)}dP_{\fx_n}^{\sigma_n}(x) + |V_{n}| \int{\phi(x)}dP_{\fx_n}^{\sigma_n}(x) + o(|V_n|),
\end{align*}
with $f(\epsilon) \to 0$ as $\epsilon \to 0$. Dividing by $|V_n|$, we obtain that
$$
\frac{\log Z_n}{|V_n|}		\approx_{f(\epsilon)}	\int{(I_\mu^{\past}(x) + \phi(x))}dP_{\fx_n}^{\sigma_n}(x) + o(1),
$$
and, since $I^\past_{\mu}$ and $\phi$ are continuous, and $\{P_{\fx_n}^{\sigma_n}\}_n$ converges in a weak* sense to $\nu$, taking the limit in $n$, we conclude that
$$
p_\Sigma(\Gamma \acts X, \phi) = \lim_{n \to \infty} \frac{\log Z_{n}}{|V_n|} \approx_{f(\epsilon)} \int{(I^\past_{\mu} + \phi)}d\nu.
$$

Then $p_\Sigma(\Gamma \acts X, \phi) \approx_{f(\epsilon)} \int{(I^\past_{\mu} + \phi)}d\nu$, and since $\epsilon$ was arbitrary, we conclude.

\end{proof}

\section{Applications}
\label{sec9}

\subsection{Percolative entropy}

Notice that, due to Corollary \ref{cor:nonnegative}, under the TSSM property assumption, we have that $h_\Sigma(\Gamma \acts (X,\mu)) \neq -\infty$. Considering this, we have the following corollary, which recovers results from Alpeev \cite{alpeev2017random} for the full shift $A^\Gamma$ and Austin and Podder \cite{austin2018gibbs} for trees, both with respect to the percolation past $\past_{\mathrm{perc}}$.

\begin{corollary}
\label{cor:perc}
If $X$ satisfies the TSSM property and there is a unique Gibbs measure $\mu$ for $\phi$ that satisfies SSM, then
$$
h_{\Sigma}(\Gamma \acts (X,\mu)) = \int{I^\past_{\mu}}d\mu = \E_\past H_\mu[\alpha \vert \alpha^{\past^{1_\Gamma}}],
$$
for every invariant random past $\past$ and for every sofic approximation $\Sigma$ that can be ordered relative to $\past$.
\end{corollary}

\begin{proof}
Since $h_\Sigma(\Gamma \acts (X,\mu)) \neq -\infty$, by the variational principle and Theorem \ref{thm:main},
$$
h_{\Sigma}(\Gamma \acts (X,\mu)) + \int{\phi}d\mu = p_{\Sigma}(\Gamma \acts X,\phi) = \int{(I_\mu^\past + \phi)}d\mu,
$$
and, by Tonelli's theorem,
$$
h_{\Sigma}(\Gamma \acts (X,\mu)) = \int{I_\mu^\past}d\mu = \int{\E_\past I_\mu}d\mu = \E_\past [\int{I_\mu}d\mu] = \E^\past H_\nu[\alpha \vert \alpha^{\past^{1_\Gamma}}].
$$
\end{proof}

\begin{remark}
In \cite{alpeev2017random}, it is also considered the \emph{attractive Gibbs structure} condition, which is a property based on an \emph{FKG order} of the configuration space, that is sufficient for obtaining Corollary \cite{alpeev2017random} under the uniqueness assumption but without requiring SSM.
\end{remark}

\begin{remark}
In general, if $\{\mu_n\}_n$ only locally weak* converges to $\mu$, it holds that
$$
h_{\Sigma}(\Gamma \acts (X,\mu)) \leq \E_{\past_{\mathrm{perc}}} H_\mu[\alpha \vert \alpha^{\past_{\mathrm{perc}}^{1_\Gamma}}].
$$
This is observed in \cite{austin2018gibbs}. In addition, it is known that $\E_{\past_{\mathrm{perc}}} H_\mu[\alpha \vert \alpha^{\past_{\mathrm{perc}}^{1_\Gamma}}]$ provides an upper bound for the so-called \emph{Rokhlin entropy}, which is at the same time an upper bound for the sofic entropy of $\mu$ (see \cite{seward2016weak,alpeev2017random}).
\end{remark}

\subsection{Amenable case}

Suppose that $\Gamma$ is amenable. Then, the usual Kolmogorov-Sinai entropy of $\Gamma \acts (X,\nu)$ coincides with the sofic entropy $h_\Sigma(\Gamma \acts (X,\nu))$ for every $\nu \in \Prob(X,\Gamma)$ and for every sofic approximation $\Sigma$ (see \cite{1-kerr}). Since the Kolmogorov-Sinai entropy is never negative and since for every invariant random past $\past$ any sofic approximation induced by a F{\o}lner sequence can be always ordered relative to $\past$ (see Example \ref{exmp:amen}), we have the following corollary.

\begin{corollary}
Let $\Gamma$ be a countable amenable group. Suppose that $X$ satisfies the TSSM property and there is a unique Gibbs measure $\mu$ for $\phi$ that satisfies SSM. Then, for every sofic approximation $\Sigma$,
$$
p_{\Sigma}(\Gamma \acts X,\phi) = \int{(I^\past_{\mu} + \phi)}d\nu,
$$
for every invariant random past $\past: \Gamma \to 2^\Gamma$ and for all $\nu \in \Prob(X,\Gamma)$.
\end{corollary}

\subsection{Independence of sofic approximation}

By Theorem \ref{thm:main}, since every sofic approximation can be ordered relative to the percolation past $\past_{\mathrm{perc}}$, we get that
$$
p_{\Sigma}(\Gamma \acts X,\phi) = \int{(I^{\past_{\mathrm{perc}}}_{\mu} + \phi)}d\mu.
$$

In particular, we observe that the right-hand side does not depend on $\Sigma$. Moreover, since
$$
\frac{\log Z_n}{|V_n|} = \int{(I_\mu^{\past_{\mathrm{perc}}}(x) + \phi(x))}dP_{\fx_n}^{\sigma_n}(x) + o(1),
$$ 
if we take limits to both sides, we have that the right-hand side converges due to the weak* convergence $P_{\fx_n}^{\sigma_n}$. Therefore,
$$
p_{\Sigma}(\Gamma \acts X,\phi) = \lim_n \frac{\log Z_n}{|V_n|},
$$
i.e., we can replace the limit superior by a limit. Considering this, we have the following corollary.

\begin{corollary}
\label{cor:indep}
If $X$ satisfies the TSSM property and $\mu$ satisfies SSM, then $p_\Sigma(\Gamma \acts X, \phi)$ is independent of the sofic approximation $\Sigma$. Moreover, it holds that
$$
p_{\Sigma}(\Gamma \acts X,\phi) = \lim_n \frac{\log Z_n}{|V_n|}	\quad	\text{ and }	\quad	h_{\Sigma}(\Gamma \acts (X,\mu)) = \lim_n \frac{H(\mu_n)}{|V_n|},
$$
i.e., we can replace the limit superior by a limit in both formulas. 
\end{corollary}

\begin{remark}
In particular, by Corollary \ref{cor:indep}, if $X$ satisfies TSSM and the uniform Gibbs measure satisfies SSM, then $h_\Sigma(\Gamma \acts X)$ is independent of the sofic approximation $\Sigma$.
\end{remark}

%
%


\subsection{Locality of pressure}

For the sake of concreteness, and without loss of generality, suppose that $\Gamma$ is a finitely generated group with generating set $S = \{s_1,\dots,s_d\}$. In such case, we can consider the Cayley structure
$$
\Cay(\Gamma, S) = \left<\Gamma;R_{s_1}(\Gamma),\dots,R_{s_d}(\Gamma)\right>,
$$
with $R_s(\Gamma) = \{(g,sg): g \in \Gamma\}$ for $s \in S$. Now, suppose that $X \subseteq A^\Gamma$ is a subshift of finite type and $\phi: X \to \R$ is a finite range potential. Up to recoding, every subshift of finite type and finite range potential can be expressed using a nearest-neighbor version of them. Therefore, without much loss of generality, we can suppose that
$$
X = \{x \in A^\Gamma: (x(g),x(sg)) \notin R_s(\mathbb{A}) \text{ for all } s \in S\},
$$
for $R_s(\mathbb{A}) \subseteq A \times A$, and $\phi = \sum_{s \in S} \phi_s$ with $\phi_s: X \to \R$ depending exclusively on $x_{(1_\Gamma,s)}$. We call $\mathbb{A} = \left<A; R_{s_1}(\mathbb{A} ),\dots,R_{s_d}(\mathbb{A} )\right>$ a {\bf constraint structure}. Given a Cayley structure $\Cay(\Gamma, S)$ and a constraint structure $\mathbb{A}$, and the corresponding binary relations $R_s(\Gamma) \subseteq \Gamma \times \Gamma$ and $R_s(\mathbb{A}) \subseteq A \times A$ for $s \in S$, we say that a map $x: \Gamma \to A$ is a {\bf structural homomorphism} if, for all $s \in S$,
$$
(g,h) \in R_s(\Gamma) \implies (x(g),x(h)) \in R_s(\mathbb{A} ).
$$

We denote by $\Hom(\Gamma,\mathbb{A})$ the set of structural homomorphisms between  $\Cay(\Gamma, S)$ and $\mathbb{A}$. Then,
\begin{enumerate}
\item we can write $X = \Hom(\Gamma,\mathbb{A})$ for every nearest-neighbor SFT;
\item structural homomorphisms are a generalization of graph homomorphisms; and
\item every SFT can be expressed using a similar formalism, possibly by allowing higher arity relations.
\end{enumerate}

In \cite{3-briceno} it was given a similar treatment of subshifts and there were provided sufficient and necessary conditions on $\mathbb{A}$ for $X = \Hom(\Gamma,\mathbb{A})$ to satisfy the TSSM property. Among those conditions, a very basic one is the existence of a {\bf safe symbol}, this is to say, an element $\textbf{0} \in A$ such that, for all $a,b \in A$,
$$
(a,b) \in R_s(\mathbb{A})	\implies (\textbf{0} ,b),  (a,\textbf{0}) \in R_s(\mathbb{A}).
$$

In simple words, $\textbf{0}$ is a symbol that can replace the appearance of any other symbol in any position without violating the constraints of $X$. There are many classical models in statistical physics that involve a safe symbol (e.g., \emph{hardcore model}, \emph{Widom-Rowlinson model}) and whenever the support $X$ has a safe symbol (or weaker properties like the \emph{unique maximal configuration property} in \cite{4-briceno}), we can find a potential $\phi$ inducing a Gibbs measure that satisfies SSM, so all the results presented in this paper apply. In particular, it is not difficult to see that the Dirac measure $\delta_{\textbf{0}^\Gamma}$ supported on the fixed point $\textbf{0}^\Gamma$ has always zero entropy with respect to any sofic approximation $\Sigma$ and, in particular, $h_\Sigma(\Gamma \acts (X,\delta_{\textbf{0}^\Gamma})) \neq -\infty$. Therefore, by Theorem \ref{thm:main}, we have that
$$
p_{\Sigma}(\Gamma \acts X,\phi) = \int{(I^{\past_\mathrm{perc}}_{\mu} + \phi)}d\delta_{\textbf{0}^\Gamma} = I^{\past_\mathrm{perc}}_{\mu}(\textbf{0}^\Gamma) + \phi(\textbf{0}^\Gamma).
$$

Now, suppose that we have two finitely generated groups $\Gamma_1$ and $\Gamma_2$ with generating sets $S_1$ and $S_2$, respectively, and a constraint structure $\mathbb{A}$. For $i=1,2$, consider the subshift $X_i = \Hom(\Gamma_i,\mathbb{A})$ and a potential $\phi_i: X_i \to \R$ that induces a Gibbs measure $\mu_i$ which satisfies SSM with decay rate $\beta_i$. Now, if we suppose that there exists $r \in \N$ such that $B(1_{\Gamma_1},r+1) \cong B(1_{\Gamma_2},r+1)$, i.e., the $(r+1)$-balls in $G(\Gamma_1,S_1)$ and $G(\Gamma_2,S_2)$ are isomorphic, it follows that
$$
\begin{array}{lll}
&&p_{\Sigma}(\Gamma_1 \acts X_1,\phi_1)	\\
&	=	&	I^{\past_{\mathrm{perc},1}}_{\mu_1}(\textbf{0}^{\Gamma_1}) + \phi_1(\textbf{0}^{\Gamma_1})	\\
								&	=	&	-\E_{\past_{\mathrm{perc},1}}\log\mu_1([(\textbf{0}^{\Gamma_1})_{1_{\Gamma_1}}] \vert [(\textbf{0}^{\Gamma_1})_{\past^{1_{\Gamma_1}}_{\mathrm{perc},1}}]) + \phi_1(\textbf{0}^{\Gamma_1})	\\
								&	\approx_{\frac{\beta_1(r)}{c_{\phi_1}}}	&	-\E_{\past_{\mathrm{perc},1}}\log\mu_1([(\textbf{0}^{\Gamma_1})_{1_{\Gamma_1}}] \vert [(\textbf{0}^{\Gamma_1})_{(\past^{1_{\Gamma_1}}_{\mathrm{perc},1} \cap B(1_{\Gamma_1},r)) \cup \partial B(1_{\Gamma_1},r))}])	+ \phi_1(\textbf{0}^{\Gamma_1})\\
								&	=	&	-\E_{\past_{\mathrm{perc},2}}\log\mu_2([(\textbf{0}^{\Gamma_2})_{1_{\Gamma_2}}] \vert [(\textbf{0}^{\Gamma_2})_{(\past^{1_{\Gamma_2}}_{\mathrm{perc},2} \cap B(1_{\Gamma_2},r)) \cup \partial B(1_{\Gamma_2},r))}]) + \phi_2(\textbf{0}^{\Gamma_2})	\\		
								&	\approx_{\frac{\beta_2(r)}{c_{\phi_2}}}	&	-\E_{\past_{\mathrm{perc},2}}\log\mu_2([(\textbf{0}^{\Gamma_2})_{1_{\Gamma_2}}] \vert [(\textbf{0}^{\Gamma_2})_{\past^{1_{\Gamma_2}}_{\mathrm{perc},2}}]) + \phi_2(\textbf{0}^{\Gamma_2})	\\
								&	=	&	I^{\past_{\mathrm{perc},2}}_{\mu_2}(\textbf{0}^{\Gamma_2}) + \phi_2(\textbf{0}^{\Gamma_2}) + \phi_2(\textbf{0}^{\Gamma_2})	\\
								&	=	&	p_{\Sigma}(\Gamma_2 \acts X_2,\phi_2),
\end{array}
$$
since $\past_{\mathrm{perc},1}$ and $\past_{\mathrm{perc},2}$ ---the corresponding percolation pasts in each group--- have the same distribution on $B(1_{\Gamma_1},r+1) \cong B(1_{\Gamma_2},r+1)$ up to isomorphism and, considering the Markov property, $\mu_1(\cdot \vert [\textbf{0}^{\partial B(1_{\Gamma_1},r))}])$ and $\mu_2(\cdot \vert [\textbf{0}^{\partial B(1_{\Gamma_2},r))}])$ are the same measure (again, up to isomorphism). Therefore,
$$
|p_{\Sigma}(\Gamma_1 \acts X_1,\phi_1) - p_{\Sigma}(\Gamma_2 \acts X_2,\phi_2)|		\leq	\frac{\beta_1(r)}{c_{\phi_1}} + \frac{\beta_2(r)}{c_{\phi_2}}.
$$

For many models it is possible to have a uniform control on $\beta_i$ and $c_{\phi_i}$ for a large family of groups and graphs. For example, in the case of the hardcore model, the constraint structure $\mathbb{A}$ is given by $A = \{0,1\}$ with $R_s(\mathbb{A}) = \{(0,0),(0,1),(1,0)\}$ for every $s \in S$, and the potential, by $\phi(x) = x(1_\Gamma) \log \lambda$, where $\lambda > 0$. Then, given $\Delta \in \N$, there exists a value 
$$
\lambda_c(\Delta) := \frac{(\Delta-1)^{\Delta-1}}{(\Delta-2)^\Delta},
$$ 
such that every hardcore model on a finitely generated group with $|S| \leq \Delta$ and $\lambda \leq \lambda_c(\Delta)$ satisfies SSM with the same decay rate $\beta$ and with the uniform bound $c(\phi)$ (see \cite{1-weitz}). The value $\lambda_c(\Delta)$ corresponds to the \emph{critical activity} of the hardcore model on the infinite $\Delta$-regular tree (see \cite{1-weitz}) and it is optimal in such graph, although there are improvements in terms of the connective constant for other graphs.

Then, in such families, sofic pressure is a \emph{local quantity}, this is to say, for every $\epsilon > 0$, there exists $r \in \N$ such that
$$
B(1_{\Gamma_1},r) \cong B(1_{\Gamma_2},r) \implies |p_{\Sigma}(\Gamma_1 \acts X_1,\phi_1) - p_{\Sigma}(\Gamma_2 \acts X_2,\phi_2)| < \epsilon.
$$

We recommend to check \cite{grimmett2018locality} for an example of locality in the case of connective constants of graphs and \cite{1-sinclair} for an intimate relationship between connective constants and SSM in the hardcore model. By combining these results we have that, under the technical conditions described in \cite{grimmett2018locality}, if two Cayley graphs are isomorphic when restricted to a large ball, then their connective constants are close and, therefore, if one satisfies SSM, the other will do as well, since SSM in the hardcore model holds for all graphs with low enough connective constant. Next, by our result, the corresponding pressures will be close in value. It is possible to generalize this to the case of almost transitive graphs, as discussed in the next subsection.

\subsection{Trees of self-avoiding walks representation}

Recently, in \cite{5-briceno}, it was developed a method to represent the pressure ---or the \emph{free energy}--- of the hardcore model on any almost transitive amenable graph, provided $\lambda$ is small enough. The method is based on an infinite version of the \emph{trees of self-avoiding walks} technique used in \cite{1-weitz} for approximating the partition function of the hardcore model on finite graphs. Considering the results in \cite{1-briceno} developed for the amenable case and Theorem \ref{thm:main}, it is possible to see that the following holds in the sofic case:
\begin{enumerate}
\item for every sofic group $\Gamma$ acting freely and almost transitively on the set of vertices $V(G)$ of a graph $G$, we can find a subshift $X$ such that $\Gamma \acts X$ is conjugated to the natural action of $\Gamma$ on the set $\Hom(G,\mathbb{A})$ of independent sets of $G$ (see \cite[Section 9.3]{5-briceno});
\item there exists $\phi_\lambda: X \to \R$ such that, due to Theorem \ref{thm:main}, for every invariant random past $\past$ and for every sofic approximation $\Sigma$ that can be ordered relative to $\past$, we can write the free energy as
$$
p_{\Sigma}(\Gamma \acts X,\phi_\lambda) = I^{\past}_{\mu}(0^{\Gamma}) + \phi_\lambda(0^{\Gamma}) = -\E_{\past}\log\mu([(0^{\Gamma})_{1_{\Gamma}}] \vert [(0^{\Gamma})_{\past^{1_{\Gamma}}}]),
$$
where $p_{\Sigma}(\Gamma \acts X,\phi_\lambda)$ is independent of $\Sigma$;
\item if the maximum degree of $G$ is bounded from above by $\Delta$ and $\lambda < \lambda_c(\Delta)$, by applying exactly the same method in \cite[Theorem 7.6]{5-briceno}, we obtain that
$$
p_{\Sigma}(\Gamma \acts X,\phi_\lambda) = -\sum_{i=1}^{|G / \Gamma|} \E_{\past} \log \mu_{T_{\mathrm{SAW}}\left(G_{i}(\past, v_{i})\right), \lambda}\left(\left[0^{\rho_{i}}\right]\right),
$$
where
\begin{itemize}
\item $|G / \Gamma|$ is the size of a (or any) fundamental domain $U_0$ of $\Gamma \acts V(G)$;
\item the vertex $v_i$ belongs to $U_0$ for all $i = 1,\dots,|G / \Gamma|$;
\item $G_{i}(\past, v_{i})$ is the (maybe random) subgraph of $G$ induced by the set of vertices
$$
V(G) \setminus \left(\past^{1_\Gamma} U_{0} \cup \left\{v_{1}, \dots, v_{i-1}\right\}\right);
$$
\item $T_{\mathrm{SAW}}\left(G_{i}(\past, v_{i})\right)$ is the tree of self-avoiding walks of $G_{i}(\past, v_{i})$ starting at $v_i$;
\item $\mu_{T_{\mathrm{SAW}}\left(G_{i}(\past, v_{i})\right), \lambda}$ is the unique Gibbs measure on $T_{\mathrm{SAW}}\left(G_{i}(\past, v_{i})\right)$; and
\item $\rho_{i}$ is the root of $T_{\mathrm{SAW}}\left(G_{i}(\past, v_{i})\right)$.
\end{itemize}

In simple words, the previous result says that whenever $\lambda$ is small enough, we can write the pressure of the hardcore model on a graph $G$ with a sofic subgroup $\Gamma$ of automorphisms acting freely and almost transitively on its vertices $V(G)$ in terms of the occupation probability at the roots of some particular trees of self-avoiding walks $T_{\mathrm{SAW}}\left(G_{i}(\past, v_{i})\right)$ on (maybe random) subgraphs $G_{i}(\past, v_{i})$ of $G$. In \cite{5-briceno}, it is also explained how to exploit these techniques in the amenable case for developing efficient approximation algorithms for $p_{\Sigma}(\Gamma \acts X,\phi_\lambda)$ and for working with any subshift of finite type that has a safe symbol, but the results extended naturally to the sofic case after considering Theorem \ref{thm:main}.
\end{enumerate}



\section*{Acknowledgements}

I would like to thank Andrei Alpeev, Tim Austin, and Brian Marcus for helpful discussions.



\bibliographystyle{abbrv}
\bibliography{biblioKP}

\end{document}